\documentclass[a4paper]{article}

\usepackage[english]{babel}
\usepackage[utf8x]{inputenc}
\usepackage[T1]{fontenc}
\newcommand{\nm}{\noalign{\smallskip}}
\newcommand{\ds}{\displaystyle}
\usepackage{graphicx}
\usepackage{caption, subcaption}
\usepackage[title]{appendix}
\usepackage{tikz}
\usepackage{pgfplots}
\usepackage{cutwin}
\graphicspath{ {MATLAB/} }
\usetikzlibrary{intersections,decorations.pathreplacing,decorations.markings,calc,angles,quotes,arrows.meta}
\usepackage[a4paper,top=2.9cm,bottom=2cm,left=2.5cm,right=2.5cm,marginparwidth=1.75cm]{geometry}

\usepackage{amsmath}
\usepackage{amsthm}
\usepackage{amssymb}
\usepackage{hyperref}
\usepackage{cleveref}
\usepackage{epstopdf}
\makeatletter
\newcommand{\neutralize}[1]{\expandafter\let\csname c@#1\endcsname\count@}
\makeatother
\usepackage{comment}
\usepackage{todonotes}

\newtheorem{thm}{Theorem}
\newtheorem{lem}[thm]{Lemma}
\newtheorem{prop}[thm]{Proposition}
\theoremstyle{definition}
\newtheorem{rmk}[thm]{Remark}
\numberwithin{equation}{section}
\numberwithin{thm}{section}
\usepackage{graphicx}
\usepackage{caption, color}
\usepackage{float}
\theoremstyle{definition}

\newcommand{\Z}{\mathbb{Z}}
\newcommand{\N}{\mathbb{N}}
\newcommand{\R}{\mathbb{R}}
\newcommand{\Ro}{\mathcal{R}}
\newcommand{\C}{\mathcal{C}}
\newcommand{\A}{\mathcal{A}}
\newcommand{\B}{\mathcal{B}}
\newcommand{\F}{\mathcal{F}}
\renewcommand{\H}{\mathcal{H}}
\newcommand{\p}{\partial}
\renewcommand{\L}{\mathcal{L}}
\renewcommand{\S}{\mathcal{S}}

\newcommand{\K}{\mathcal{K}}

\renewcommand{\epsilon}{\varepsilon}
\newcommand{\dx}{\: \mathrm{d}}

\renewcommand{\b}[1]{\textbf{#1}}
\newcommand{\eqnref}[1]{(\ref {#1})}
\def\nm{\noalign{\medskip}}

\newcommand{\ie}{\textit{i.e.}}
\newcommand{\iu}{\mathrm{i}\mkern1mu}

\title{Topologically protected edge modes in one-dimensional chains of subwavelength resonators}

\author{
	Habib Ammari\thanks{\footnotesize Department of Mathematics, 
		ETH Z\"urich, 
		R\"amistrasse 101, CH-8092 Z\"urich, Switzerland (habib.ammari@math.ethz.ch, bryn.davies@sam.math.ethz.ch, erik.orvehed.hiltunen@sam.math.ethz.ch).}\and Bryn Davies\footnotemark[1]  \and Erik Orvehed Hiltunen\footnotemark[1]  \and Sanghyeon Yu\thanks{\footnotesize Department of Mathematics, Korea University, Seoul 02841, S. Korea (sanghyeon\_yu@korea.ac.kr)}}

\date{}

\begin{document}
	\maketitle
	
	\begin{abstract}
		The goal of this paper is to advance the development of wave-guiding subwavelength crystals by developing designs whose properties are stable with respect to imperfections in their construction. In particular, we make use of a locally resonant subwavelength structure, composed of a chain of high-contrast resonators, to trap waves at deep subwavelength scales. We first study an infinite chain of subwavelength resonator dimers and define topological quantities that capture the structure's wave transmission properties. Using this for guidance, we design a finite crystal that is shown to have wave localization properties, at subwavelength scales, that are robust with respect to random imperfections.
	\end{abstract}
\vspace{0.5cm}
	\noindent{\textbf{Mathematics Subject Classification (MSC2000):} 35J05, 35C20, 35P20.
		
\vspace{0.2cm}

\vspace{0.5cm}	
%
	
	\section{Introduction}
	The ability to manipulate and guide the propagation of waves on subwavelength scales is important for many different physical applications. In the contexts of nanophotonics and nanophononics, subwavelength crystalline structures have, in particular, been shown to have desirable properties. Here, \textit{subwavelength} means that the size of the repeating unit cell is several magnitudes smaller than the incident wavelengths. It was recently shown, for example, that one can design subwavelength crystals with low ranges of frequencies that cannot propagate (known as \emph{subwavelength band gaps}) \cite{bandgap} and can localize (or trap) specific frequencies at subwavelength scales by introducing a defect \cite{defectSIAM}. However, one limitation of such designs is that their properties are often very sensitive to imperfections in the crystal's structure. In order to be able to feasibly manufacture wave-guiding devices, it is important that we are able to design subwavelength crystals that exhibit stability with respect to geometric errors.
	
	We take inspiration from quantum mechanics. So-called \emph{topological insulators} have been extensively studied for their electronic properties, in the setting of the Schrödinger operator \cite{drouot2, drouot1,fefferman,fefferman2,fefferman3,fefferman4, lee-thorp}. The principle that underpins the design of these structures is that one is able to define topological invariants which capture the crystal's wave propagation properties. Then, if part of a crystalline structure is replaced with an arrangement that is associated with a different value of this invariant, not only will certain frequencies be localized to the interface (as predicted by the classical theory for crystals with defects) but this behaviour will be stable with respect to imperfections. These eigenmodes are known as \emph{edge modes} and we say that they are \emph{topologically protected} to refer to their robustness.
	
	One of the most classical examples from quantum mechanics is the well-known Su-Schrieffer-Heeger (SSH) model \cite{SSH}. Originally introduced to study the electrical properties of polyacetylene, the SSH model consists of a chain of atoms arranged as dimers (similar to that depicted in \Cref{fig:SSH}). In the case of one-dimensional crystals such as this, the natural choice of topological invariant is the Zak phase \cite{zak}. Qualitatively, a non-zero Zak phase means that the crystal has undergone \emph{band inversion}, meaning that at some point in the Brillouin zone the monopole/dipole nature of the first/second Bloch eigenmodes has swapped. In this way, the Zak phase captures the crystal's wave propagation properties. The Zak phased was measured experimentally by \cite{zak_experiment}, and in the SSH model this can take two discrete values depending on whether the atoms in each dimer are closer to each other than they are to the next dimer in the chain. In higher dimensional crystals, topological indices are similarly dependent on the symmetry of the crystals \cite{top10}. If one takes two SSH chains with different Zak phases and joins half of one chain to half of the other to form a new crystal, this crystal will exhibit a topologically protected edge mode at the interface. This principle is known as the \emph{bulk-boundary correspondence} in quantum settings \cite{drouot2, drouot4,graf2013bulk2d,graf2018bulk,graf2018bulk2d,bulkbdy,prodan}. Here, the term \emph{bulk} is used to refer to parts of a crystal that are away from an edge (and so resemble an infinite, defect-free crystalline structure).
	
	Understanding \emph{why} topologically protected edge modes are stable to local perturbations is subtle, and doing so precisely is very much an open question. It can be argued that, due to (chiral) symmetry, these crystals not only have band gaps but the frequencies associated with edge modes occur in the centre of the band gap. We call them \textit{midgap} frequencies. This is in sharp contrast to conventional, unprotected defect frequencies, which typically emerge from the edge of the band gap \cite{defectSIAM}. As a result, a small imperfection in the subwavelength crystal will not be able to move a topologically protected frequency out of the band gap, while an unprotected frequency is often lost amongst the bulk frequencies. Moreover, if the perturbation preserves the crystal's symmetry, the frequency of the edge mode will be very stable, experiencing much smaller variations compared to the other subwavelength resonant frequencies. These effects are typically explained as a consequence of the different topological properties on either side of the edge (see, for example, \cite{top_review} for a review of topological phases in acoustic systems).
	
	Subwavelength topological photonic and phononic crystals, based on locally resonant crystalline structures with large material contrasts, have been studied both numerically and experimentally in \cite{top_subwavelength,top_subwavelength2,Yves2,Yves1,Yves3}. Subwavelength crystals allow for the manipulation and localization of waves on very small spatial scales and are therefore very useful in physical applications, especially situations where the operating wavelengths are very large. Recently, topological properties of acoustic waves in SSH chains and honeycomb lattices of subwavelength resonators have been numerically and experimentally explored \cite{add2,add1,add4,add3}. In this work, we study a subwavelength crystal exhibiting a topologically non-trivial band gap. The crystal consists of a chain of subwavelength resonators arranged as dimers, similar to the SSH model (see \Cref{fig:SSH,,fig:finite}). Wave propagation in the resonant structure is modelled by a high-contrast Helmholtz problem. High material contrasts are an essential prerequisite for the existence of resonant behaviours on subwavelength scales \cite{first, minnaert}. Such problems arise naturally in the context of both nanophotonic and nanophononic structures \cite{brynCochlea, first, MaCMiPaP}. Around this frequency, a single resonator in free-space scatters waves with a greatly enhanced amplitude. If a second resonator is introduced, coupling interactions will occur giving a system that has both monopole and dipole resonance modes \cite{doublenegative}. This pattern continues for larger systems \cite{brynCochlea}. 
	
	Initially, in \Cref{sec:inf_chain}, we set out to study the bulk properties of an infinitely periodic chain of subwavelength resonator dimers. Using Floquet-Bloch theory, we are able to analytically derive the resonant frequencies and associated eigenmodes of this crystal, and further prove that there exists a non-trivial band gap. Motivated by the use of the Zak phase in quantum mechanics, as well as the work of \cite{top_acoustic_SSH, top_review, pocock2018, zhao2018} in photonics and phononics, we define a topological invariant which we will also refer to as the Zak phase. We prove that the Zak phase takes different values for different geometries and in the \emph{dilute regime} (that is, when the distance between the resonators is an order of magnitude greater than their size) we give explicit expressions for its value. Guided by this knowledge of how the infinite (bulk) chains behave, in \Cref{sec:finite_crystals} we design a finite chain of resonator dimers that has a topologically protected edge mode. This configuration takes inspiration from the bulk-boundary correspondence in the SSH model by introducing an interface, on either side of which the resonator dimers can be associated with different Zak phases thus creating a topologically protected edge mode.
	
	In the quantum mechanical SSH model, the standard approach is to consider the \emph{tight-binding approximation}. In this set-up, the Hamiltonian corresponding to the continuous differential problem is simplified by assuming that each particle only interacts with the surrounding crystal in a limited way that is easy to describe. This simplification gives a discrete approximation to the problem. Often, this is combined with a \emph{nearest-neighbour approximation}, where long-range interactions are neglected, enabling explicit computations of the band structure. In this work, we prove that our system can be approximated by a discrete system, which captures all the interactions in full and has rigorous error estimates. In the dilute regime, we quantify the decay of the interactions and conclude that non-negligible interactions occur also for resonators separated by several unit cells. Since the edge modes are protected due to \emph{chiral symmetry}, which is only  present here under the nearest-neighbour approximation, we expect the midgap frequencies to be \emph{approximately} stable with respect to errors which preserve this symmetry.
	
	
	Finally, we conduct a fully-coupled numerical study of our finite chain of resonator dimers. This is based on an approach similar to that developed in \cite{brynCochlea}. We show that the crystal can exhibit topologically protected subwavelength edge modes in both the dilute and non-dilute regimes. Moreover, we study the stability of the midgap frequencies with respect to symmetry-preserving geometric errors. We show that, while the midgap frequency experiences variations (which is not the case under the nearest-neighbour approximation), these are much smaller than those seen by the band frequencies and the edge mode remains localized in the middle of the band gap even for very large geometric errors. We also make the comparison with a classical, unprotected, defect mode, similar to that studied in \cite{defectSIAM}. We show that the new subwavelength crystal exhibits a mode with a similar degree of localization but with greatly improved stability with respect to errors.

 	\section{Preliminaries}
 	In this section, we briefly review the layer potential operators and Floquet-Bloch theory that will be used in the subsequent analysis. More details on this material can be found in, for example, \cite{MaCMiPaP}. We also briefly review topological properties of the band structure.
 	
 	\subsection{Layer potential techniques} \label{sec:layerpot}
 	Let $D\in \R^3$ be a bounded, multiply-connected domain with $N$ simply-connected components $D_i$. Further, suppose that there exists some $0<s<1$ so that $\p D_i$ is of class $C^{1,s}$ for each $i=1,\ldots,N$. Let $G^0$ and $G^k$  be the Laplace and (outgoing) Helmholtz Green's functions, respectively, defined by
 	$$
 	G^k(x,y) := -\frac{e^{\iu k|x-y|}}{4\pi|x-y|}, \quad x,y \in \R^3, k\geq 0.
 	$$
 	
 	We introduce the single layer potential $\S_{D}^k: L^2(\partial D) \rightarrow H_{\textrm{loc}}^1(\R^3)$, defined by
 	\begin{equation*}
 	\S_D^k[\phi](x) := \int_{\partial D} G^k(x,y)\phi(y) \dx \sigma(y), \quad x \in \R^3.
 	\end{equation*}
 	Here, the space $H_{\textrm{loc}}^1(\R^3)$ consists of functions that are square integrable on every compact subset of $\R^3$ and have a weak first derivative that is also square integrable. It is well-known that $\S_D^0: L^2(\p D) \rightarrow H^1(\p D)$ is invertible, where $H^1(\p D)$ is the set of functions that are square integrable on $\p D$ and have a weak first derivative that is also square integrable.
 	
 	We also define the Neumann-Poincar\'e operator $\K_D^{k,*}: L^2(\partial D) \rightarrow L^2(\partial D)$ by
 	\begin{equation*}
 	\K_D^{k,*}[\phi](x) := \int_{\partial D} \frac{\partial }{\partial \nu_x}G^k(x,y) \phi(y) \dx \sigma(y), \quad x \in \partial D,
 	\end{equation*}
 	where $\partial/\partial \nu_x$ denotes the outward normal derivative at $x\in\p D$.
 	
 	The following relations, often known as \emph{jump relations}, describe the behaviour of $\S_D^k$ on the boundary $\partial D$ (see, for example, \cite{MaCMiPaP}):
 	\begin{equation}
 	\S_D^k[\phi]\big|_+ = \S_D^k[\phi]\big|_-,
 	\end{equation}
 	and
 	\begin{equation}
 	\frac{\partial }{\partial \nu}\S_D^k[\phi]\Big|_{\pm}  =  \left(\pm\frac{1}{2} I + \K_D^{k,*}\right) [\phi],
 	\end{equation}
 	where $|_\pm$ denote the limits from outside and inside $D$.

 	

 	\subsection{Floquet-Bloch theory and quasiperiodic layer potentials}\label{sec:floquet}
 	A function $f(x)\in L^2(\R)$ is said to be $\alpha$-quasiperiodic, with quasiperiodicity $\alpha\in\mathbb{R}$, if $e^{-\iu \alpha x}f(x)$ is periodic. If the period is $L>0$, the natural space for $\alpha$ is $Y^*:= \R / \tfrac{2\pi}{L} \Z \simeq (-\pi/L, \pi/L]$. $Y^*$ is known as the \textit{first Brillouin zone}. Given a function $f\in L^2(\R)$, the Floquet transform is defined as
 	\begin{equation}\label{eq:floquet}
 	\F[f](x,\alpha) := \sum_{m\in \Z} f(x-Lm) e^{\iu L\alpha m}.
 	\end{equation}
 	$\F[f]$ is always $\alpha$-quasiperiodic in $x$ and periodic in $\alpha$. Let $Y_0 = [-L/2,L/2)$ be the unit cell. The Floquet transform is an invertible map $\F:L^2(\R) \rightarrow L^2(Y_0\times Y^*)$, with inverse (see, for instance, \cite{MaCMiPaP, kuchment})
 	\begin{equation*}
 	\F^{-1}[g](x) = \frac{1}{2\pi}\int_{Y^*} g(x,\alpha) \dx \alpha, \quad x\in \R,
 	\end{equation*}
 	where $g(x,\alpha)$ is the quasiperiodic extension of $g$ for $x$ outside of the unit cell $Y_0$.
 	
 	We will consider a three-dimensional problem which is periodic in one dimension. Define the unit cell $Y$ as $Y := Y_0\times \R^2$. We define the quasiperiodic Green's function $G^{\alpha,k}(x,y)$ as the Floquet transform of $G^k(x,y)$ in the first dimension of $x$, \ie{},
 	$$G^{\alpha,k}(x,y) := -\sum_{m \in \Z} \frac{e^{\iu k|x-y-(Lm,0,0)|}}{4\pi|x-y-(Lm,0,0)|}e^{\iu \alpha Lm}.$$	
 	
 	Let $D$ be as in the previous layer potential definitions, but assume additionally $D\Subset Y$. We define the quasiperiodic single layer potential $\mathcal{S}_D^{\alpha,k}$ by
 	$$\mathcal{S}_D^{\alpha,k}[\phi](x) := \int_{\partial D} G^{\alpha,k} (x,y) \phi(y) \dx\sigma(y),\quad x\in \mathbb{R}^3.$$
 	It is known that $\mathcal{S}_D^{\alpha,0} : L^2(\p D) \rightarrow H^1(\p D)$ is invertible if $\alpha \neq  0$ \cite{MaCMiPaP}.
 	It satisfies the jump relations
 	\begin{equation} \label{eq:jump1}
 	\S_D^{\alpha,k}[\phi]\big|_+ = \S_D^{\alpha,k}[\phi]\big|_-,
 	\end{equation}
 	and
	\begin{equation} \label{eq:jump2}
 	\frac{\p}{\p\nu} \mathcal{S}_D^{\alpha,k}[\phi] \Big|_{\pm} = \left( \pm \frac{1}{2} I +( \mathcal{K}_D^{-\alpha,k} )^*\right)[\phi]\quad \mbox{on}~ \p D,
 	\end{equation}
 	where $(\mathcal{K}_D^{-\alpha,k})^*$ is the quasiperiodic Neumann-Poincaré operator, given by
 	$$ (\mathcal{K}_D^{-\alpha, k} )^*[\phi](x):= \int_{\p D} \frac{\p}{\p\nu_x} G^{\alpha,k}(x,y) \phi(y) \dx\sigma(y).$$
 	
 	
 	
	\subsection{Band structure and topological properties} \label{sec:top}
	In this section we briefly review the topological nature of the Bloch eigenbundle. For further details, and discussions of the topological quantities involved, we refer to \cite{shortcourse,phases,kanereview}. Let $\mathcal{L}$ be a linear elliptic differential operator which is self-adjoint in $L^2(\R^3)$ and whose coefficients are periodic in one dimension. Denote by $\mathcal{L}(\alpha)$ the operator with the same differential expression but acting on $\alpha$-quasiperiodic functions. It is well-known \cite{kuchment2} that the spectrum $\sigma(\mathcal{L})$ of the original operator can be expressed in terms of the spectra $\sigma(\mathcal{L}(\alpha))$ as
	$$\sigma(\mathcal{L}) = \bigcup_{\alpha\in Y^*} \sigma(\mathcal{L}(\alpha)).$$
	This describes a \textit{band structure} of the spectrum of $\mathcal{L}$: for each $\alpha$ the spectrum $\sigma(\mathcal{L}(\alpha))$ is known to be discrete and will thus trace out bands $\sigma_n(\mathcal{L}(\alpha)), n = 1,2,\dots$, as $\alpha$ varies. The spectrum of $\mathcal{L}$ is said to have a \textit{band gap} if, for some $n$, $\max_\alpha\sigma_n(\mathcal{L}(\alpha)) < \min_\alpha \sigma_{n+1}(\mathcal{L}(\alpha))$. A band is said to be \textit{non-degenerate} if it does not intersect any other band.
	
	On a non-degenerate band, indexed by $n=1,2,\dots$, there exists a family of associated Bloch eigenmodes $\{u_n^\alpha\}_{\alpha\in Y^*}$ which we define so that they are both normalized and depend continuously on $\alpha$. Observe that the base space $Y^*$ has the topology of a circle. A natural question to ask, when considering the topological properties of a crystal, is whether properties are preserved after parallel transport around $Y^*$. In particular, a powerful quantity to study is the \textit{Berry-Simon connection} $A_n$, defined as
	$$A_n := \iu \langle u_n^\alpha, \frac{\p}{\p \alpha} u_n^\alpha \rangle.$$
	For any $\alpha_1,\alpha_2\in Y^*$, the parallel transport from $\alpha_1$ to $\alpha_2$ is $u_n^{\alpha_1}\mapsto e^{\iu \theta}u_n^{\alpha_2}$, where $\theta$ is given by
	\begin{equation*}
	\theta = \int_{\alpha_1}^{\alpha_2} A_n \dx \alpha.
	\end{equation*}
	Thus, it is enlightening to introduce the so-called \textit{Zak phase}, $\varphi_n^z$, defined as
	$$\varphi_n^{z} := \iu \int_{Y^*} \big\langle u_n^\alpha, \frac{\p }{\p \alpha} u_n^\alpha \big\rangle \dx \alpha,$$
	which corresponds to parallel transport around the whole of $Y^*$. When $\varphi_n^z$ takes a value that is not a multiple of $2\pi$, we see that the eigenmode has gained a non-zero phase after parallel transport around the circular domain $Y^*$. In this way, the Zak phase captures topological properties of the crystal. For crystals with inversion symmetry, the Zak phase is known to only attain the values $0$ or $\pi$ \cite{Zak_acoustic, zak}.
	
	\begin{rmk}
		In quantum mechanical contexts, $\langle \cdot, \cdot \rangle$ is typically chosen as the $L^2(Y)$-inner product. When working with Helmholtz scattering problems, however, this choice is not appropriate since the solutions are not elements of $L^2(Y)$. Instead, we will work with the $L^2(D)$-inner product. We will see that the behaviour on $D$ is enough to characterize non-trivial topological behaviour and capture the structure's wave propagation properties.
	\end{rmk}

	\section{Infinite, periodic chains of subwavelength resonator dimers} \label{sec:inf_chain}
	In this section, we study a periodic arrangement of subwavelength resonator dimers. This is an analogue of the SSH model. The goal is to derive a topological invariant which characterises the crystal's wave propagation properties and indicates when it supports topologically protected edge modes. The analysis here holds for a very general class of high-contrast resonator chains, requiring only two assumptions of geometric symmetry.
	
	\subsection{Problem description}
	Assume we have a one-dimensional crystal in $\R^3$ with repeating unit cell $Y := [-\frac{L}{2}, \frac{L}{2}]\times \R^2$. Each unit cell contains two resonators (often referred to as a dimer) surrounded by some background medium. Suppose the resonators together occupy the domain $D := D_1 \cup D_2$. As well as sufficient smoothness for the above layer potential operators to be well defined, we need two assumptions of symmetry for the analysis that follows. The first is that each individual resonator is symmetric in the sense that there exists some $x_1\in\mathbb{R}$ such that
	\begin{equation} \label{resonator_symmetry}
	R_1D_1 = D_1, \quad R_2D_2 = D_2,
	\end{equation}
	where $R_1$ and $R_2$ are the reflections in the planes $p_1=\{-x_1\}\times \R^2$ and $p_2=\{x_1\}\times \R^2$, respectively. We also assume that the dimer is symmetric in the sense that
	\begin{equation} \label{dimer_symmetry}
	D=-D.
	\end{equation}
	 Denote the full crystal by $\C$, that is,
	\begin{equation} \label{crystal_def}
	\C := \bigcup_{m\in \Z} \left(D + (mL,0,0)\right).
	\end{equation}
	We denote the separation of the resonators within each unit cell, along the first coordinate axis, by $d := 2x_1$ and the separation across the boundary of the unit cell by $d' := L - d$. 
	
	\begin{figure}[tbh]
		\centering
		\begin{tikzpicture}[scale=2]
		\begin{scope}
		\draw[dashed, opacity=0.5] (-0.5,0.85) -- (-0.5,-1);
		\draw[dashed, opacity=0.5]  (1.8,0.85) -- (1.8,-1)node[yshift=4pt,xshift=-7pt]{};
		\draw[{<[scale=1.5]}-{>[scale=1.5]}, opacity=0.5] (-0.5,-0.6) -- (1.8,-0.6)  node[pos=0.5, yshift=-7pt,]{$L$};
		\draw  plot [smooth cycle] coordinates {(-0.1,0.1) (0.2,0) (0.5,0.1) (0.3,-0.4) (0.1,-0.4)} node[xshift=-13pt, yshift=10pt]{$D_1$};
		\draw  plot [smooth cycle] coordinates {(0.8,-0.1) (1.1,0) (1.4,-0.1) (1.2,0.4) (1,0.4)} node[xshift=25pt, yshift=-9pt]{$D_2$};
		\draw[{<[scale=1.5]}-{>[scale=1.5]}, opacity=0.5] (0.2,0.6) -- (1.1,0.6) node[pos=0.5, yshift=-5pt,]{$d$};
		\draw[dotted,opacity=0.5] (0.2,0.7) -- (0.2,-0.8) node[at end, yshift=-0.2cm]{$p_1$};
		\draw[dotted,opacity=0.5] (1.1,0.7) -- (1.1,-0.8) node[at end, yshift=-0.2cm]{$p_2$};
		\draw[{<[scale=1.5]}-{>[scale=1.5]}, opacity=0.5] (1.1,0.6) -- (2.5,0.6) node[pos=0.6, yshift=-5pt,]{$d'$};
		\end{scope}
		\begin{scope}[xshift=-2.3cm]
		\draw  plot [smooth cycle] coordinates {(-0.1,0.1) (0.2,0) (0.5,0.1) (0.3,-0.4) (0.1,-0.4)};
		\draw  plot [smooth cycle] coordinates {(0.8,-0.1) (1.1,0) (1.4,-0.1) (1.2,0.4) (1,0.4)};
		\begin{scope}[xshift = 1.2cm]
		\draw (-1.6,0) node{$\cdots$};
		\end{scope};
		\end{scope}
		\begin{scope}[xshift=2.3cm]
		\draw  plot [smooth cycle] coordinates {(-0.1,0.1) (0.2,0) (0.5,0.1) (0.3,-0.4) (0.1,-0.4)};
		\draw  plot [smooth cycle] coordinates {(0.8,-0.1) (1.1,0) (1.4,-0.1) (1.2,0.4) (1,0.4)};
		\draw[dotted] (0.2,0.7) -- (0.2,-0.8);
		\begin{scope}[xshift = 1.1cm]
		\end{scope}
		\draw (1.7,0) node{$\cdots$};
		\end{scope}		
		\begin{scope}[yshift=0.9cm]
		\draw [decorate,opacity=0.5,decoration={brace,amplitude=10pt}]
		(-0.5,0) -- (1.8,0) node [black,midway]{};
		\node[opacity=0.5] at (0.67,0.35) {$Y$};	
		\end{scope}
		\end{tikzpicture}
		\caption{Example of a two-dimensional cross-section of a chain of subwavelength resonators satisfying the symmetry assumptions \eqref{resonator_symmetry}~and~\eqref{dimer_symmetry}. The repeating unit cell $Y$ contains the dimer $D_1 \cup D_2$.} \label{fig:SSH}
	\end{figure}
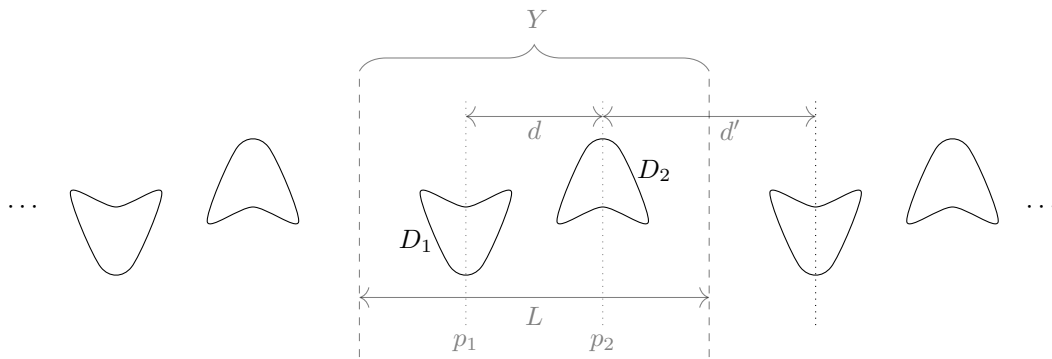	
	
	Wave propagation inside the infinite periodic structure is modelled by the Helmholtz problem
	\begin{equation} \label{eq:scattering}
	\left\{
	\begin{array} {ll}
	\ds \Delta {u}+ \omega^2 {u}  = 0 & \text{in } \R^3 \setminus \p \C, \\
	\nm
	\ds  {u}|_{+} -{u}|_{-}  =0  & \text{on } \partial \C, \\
	\nm
	\ds  \delta \frac{\partial {u}}{\partial \nu} \bigg|_{+} - \frac{\partial {u}}{\partial \nu} \bigg|_{-} =0 & \text{on } \partial \C, \\
	\nm
	\ds u(x_1,x_2,x_3) & \text{satisfies the outgoing radiation condition as } \sqrt{x_2^2+x_3^2} \rightarrow \infty.
	\end{array}
	\right.
	\end{equation}
	Here, $\omega$ is the frequency of the incident waves which is assumed to be small, such that we are in a subwavelength regime. We refer to \cite{MaCMiPaP} for the definition of the outgoing radiation condition. The material parameter $\delta$ represents the contrast between the resonators and the background.
	In order for subwavelength resonant modes to exist, we assume that $\delta$ satisfies the high-contrast condition
	\begin{equation} \label{data2}
	\delta \ll 1.
	\end{equation}
	As an example, in the case of acoustic wave propagation, $\delta = \rho_r/\rho_0$ is the density contrast between the resonator material and the background material.
	
	Let $\sigma$ be the spectrum of the operator
	$$\mathcal{L} := -\frac{1}{\delta  + (1-{\delta}) \chi_{\C}}\nabla \cdot \left(\Big(\delta  + (1-{\delta}) \chi_{\C}\Big)\nabla\right),$$
	acting on functions which satisfy the outgoing radiation condition in $x_2, x_3$. Here, $\chi_{\C}$ denotes the indicator function of the periodic crystal $\C$. Only frequencies $\omega$ such that $\omega^2 \in \sigma$ can be solutions to \eqref{eq:scattering}. Any other frequencies are not able to propagate in the material. It is worth emphasizing that, due to radiation in $x_2$- and $x_3$-directions,
the resonant frequencies are complex with negative imaginary parts. Nevertheless, as we will see in Theorem \ref{thm:char_approx_infinite}, the resonant frequencies are real at leading order so we consider only their real parts in this work.

	By applying the Floquet transform, the Bloch eigenmode $u_\alpha(x) := \F[u](x,\alpha)$ is the solution to the Helmholtz problem
	\begin{equation} \label{eq:scattering_quasi}
	\left\{
	\begin{array} {ll}
	\ds \Delta u_\alpha+ \omega^2 {u_\alpha}  = 0 &\text{in } \R^3 \setminus \p \C, \\
	\nm
	\ds  {u_\alpha}|_{+} -{u_\alpha}|_{-}  =0  & \text{on } \partial \C, \\
	\nm
	\ds  \delta \frac{\partial {u_\alpha}}{\partial \nu} \bigg|_{+} - \frac{\partial {u_\alpha}}{\partial \nu} \bigg|_{-} =0 & \text{on } \partial \C, \\
	\nm
	\ds e^{-\iu  \alpha_1 x_1}  u_\alpha(x_1,x_2,x_3)  \,\,\,&  \mbox{is periodic in } x_1, \\
	\nm
	 \ds u_\alpha(x_1,x_2,x_3)& \text{satisfies the $\alpha$-quasiperiodic outgoing radiation condition} \\ &\hspace{0.5cm} \text{as } \sqrt{x_2^2+x_3^2} \rightarrow \infty.
	\end{array}
	\right.
	\end{equation} 
	We refer to \cite{MaCMiPaP} for the definition of the $\alpha$-quasiperiodic outgoing radiation condition. We denote by $\sigma_\alpha$ the spectrum of the operator
	$$\mathcal{L}(\alpha):=-\frac{1}{\delta  + (1-{\delta}) \chi_{\C}}\nabla \cdot \left(\Big(\delta  + (1-{\delta}) \chi_{\C}\Big)\nabla\right),$$
	acting on $\alpha$-quasiperiodic functions which satisfy the outgoing radiation condition in $x_2, x_3$. As discussed in \Cref{sec:top}, we have
	$$\sigma = \bigcup_{\alpha\in Y^*} \sigma_\alpha,$$	
	which describes the band structure of the crystal. 
	
	\subsection{Analysis of quasiperiodic problem} \label{sec:formulation_quasip}
	In this section we conduct a thorough analysis of the band structure and topological properties of \eqref{eq:scattering_quasi}. We use the methods from \cite{honeycomb, bandgap} to formulate the quasiperiodic resonance problem as an integral equation. The solution $u_\alpha$ of \eqref{eq:scattering_quasi} can be represented as
	\begin{equation*} \label{eq:helm-solution_quasi}
	u_\alpha = \mathcal{S}_{D}^{\alpha,\omega} [\Psi^\alpha],
	\end{equation*}
	for some density $\Psi^\alpha \in  L^2(\p D)$. Then, using the jump relations \eqref{eq:jump1} and \eqref{eq:jump2}, it can be shown that~\eqref{eq:scattering_quasi} is equivalent to the boundary integral equation
	\begin{equation}  \label{eq:boundary_quasi}
	\mathcal{A}^\alpha(\omega, \delta)[\Psi^\alpha] =0,  
	\end{equation}
	where
	\begin{equation} \label{eq:A_quasi_defn}
	\mathcal{A}^\alpha(\omega, \delta) := -\lambda I + \left(\mathcal{K}_D^{ -\alpha,\omega}\right)^*, \quad \lambda := \frac{1+\delta}{2(1-\delta)}.
	\end{equation}
	\subsubsection{Quasiperiodic capacitance matrix}\label{subsec:cap_quasi}
	Let $V_j^\alpha$ be the solution to 
	\begin{equation} \label{eq:V_quasi}
	\begin{cases}
	\ds \Delta V_j^\alpha =0 \quad &\mbox{in } \quad Y\setminus D,\\
	\ds V_j^\alpha = \delta_{ij} \quad &\mbox{on } \quad \partial D_i,\\
	\ds V_j^\alpha(x+(mL,0,0))= e^{\iu \alpha m} V_j^\alpha(x) & \forall m \in \Z, \\
	\ds V_j^\alpha(x_1,x_2,x_3) = O\left(\tfrac{1}{\sqrt{x_2^2+x_3^2}}\right) \quad &\text{as } \sqrt{x_2^2+x_2^2}\to\infty, \text{ uniformly in } x_1,
	\end{cases}
	\end{equation}
	where $\delta_{ij}$ is the Kronecker delta.
	We then define the quasiperiodic capacitance matrix $C^\alpha=(C_{ij}^\alpha)$ by
	\begin{equation} \label{eq:qp_capacitance} 
	C_{ij}^\alpha := \int_{Y\setminus D}\overline{\nabla V_i^\alpha}\cdot\nabla V_j^\alpha  \dx x,\quad i,j=1, 2.
	\end{equation}
	We will see, shortly, that finding the eigenpairs of this matrix represents a leading order approximation to the differential problem \eqref{eq:scattering_quasi}. First, however, we show some useful properties of $C^\alpha$.

	\begin{lem} \label{lem:quasi_matrix_form}
		The matrix $C^\alpha$ is Hermitian with constant diagonal, \ie{},
		$$C_{11}^\alpha = C_{22}^\alpha \in \R, \quad C_{12}^\alpha = \overline{C_{21}^\alpha} \in \mathbb{C}.$$
	\end{lem}
	\begin{proof}
		From the definition \eqnref{eq:qp_capacitance}, it clearly follows that $C^\alpha$ is Hermitian. To show that $C_{11}^\alpha = C_{22}^\alpha$, we define the mapping $T$ by
		\begin{equation}
		(Tf)(x):= f(-x).
		\end{equation} 
		Then, thanks to the assumed symmetry of the dimer \eqref{dimer_symmetry}, it holds that $TV_1^\alpha = \overline{V_2^\alpha}$ and $TV_2^\alpha = \overline{V_1^\alpha}$. It follows that 
		\begin{align*}
		C_{11}^\alpha &= \int_{Y\setminus D} \overline{\nabla V_1^\alpha}\cdot\nabla V_1^\alpha  \dx x \\ 
		&= \int_{Y\setminus D} \overline{\nabla TV_1^\alpha}\cdot\nabla TV_1^\alpha  \dx x \\		
		&= \int_{Y\setminus D} \nabla V_2^\alpha\cdot\overline{\nabla V_2^\alpha}  \dx x \\
		&= C_{22}^\alpha.\tag*{\qedhere}
		\end{align*}
	\end{proof}
	Using the jump conditions, in the case $\alpha\neq0$, it can be shown that the capacitance coefficients $C_{ij}^\alpha$ are also given by
	$$ C_{ij}^\alpha = - \int_{\partial D_i} \psi_j^\alpha \dx \sigma,\quad i,j=1, 2,$$
	where $\psi_j^\alpha$ are defined by
	$$\psi_j^\alpha = (\S_D^{\alpha,0})^{-1}[\chi_{\p D_j}].$$

	Since $C^\alpha$ is Hermitian, the following lemma follows directly.
	\begin{lem} \label{lem:evec}
		The eigenvalues and corresponding eigenvectors of the quasiperiodic capacitance matrix are given by 
		\begin{align*}
		\lambda_1^\alpha &= C_{11}^\alpha - \left|C_{12}^\alpha \right|, \qquad
		\begin{pmatrix}
		a_1  \\ b_1
		\end{pmatrix} = \frac{1}{\sqrt{2}}\begin{pmatrix}
		- e^{\iu \theta_\alpha}  \\ 1
		\end{pmatrix}, \\
		\lambda_2^\alpha &= C_{11}^\alpha + \left|C_{12}^\alpha \right|, \qquad
		\begin{pmatrix}
		a_2  \\ b_2
		\end{pmatrix} = \frac{1}{\sqrt{2}}\begin{pmatrix}
		e^{\iu \theta_\alpha}  \\ 1
		\end{pmatrix},
		\end{align*}
		where, for $\alpha$ such that $C_{12}^\alpha\neq0$, $\theta_\alpha\in[0,2\pi)$ is defined to be such that
		\begin{equation}
			e^{\iu \theta_\alpha} = \frac{C_{12}^\alpha}{|C_{12}^\alpha|}.
		\end{equation}
	\end{lem}
	
	In the dilute regime, we are able to compute asymptotic expansions for the band structure and topological properties. In this regime, we assume that the resonators can be obtained by rescaling fixed domains $B_1, B_2$ as follows:
	\begin{equation}\label{eq:dilute}
	D_1=\epsilon B_1 - \left(\frac{d}{2},0,0\right), \quad  D_2=\epsilon B_2 + \left(\frac{d}{2},0,0\right),
	\end{equation}
	for some small parameter $\epsilon > 0$. 
	
	We introduce the capacitance $\textrm{Cap}_{B}$ of the fixed domains as follows. Let $B = B_i$ for $i=1$ or $i=2$ and define
	$$
	\textrm{Cap}_{B} := -\int_{\p B}\phi_{B} \dx \sigma, 
	$$
	where $\phi_{B} := (\mathcal{S}_{B}^0)^{-1}[\chi_{\p B}]$. Due to symmetry, the capacitance is the same for the two choices $i =1, 2$. It is easy to see that, by a scaling argument, 
	\begin{equation}\label{eq:cap_scale}
	\textrm{Cap}_{\epsilon B} = \epsilon \textrm{Cap}_B.
	\end{equation}
	
	\begin{lem}\label{lem:cap_estim_quasi}
		We assume that the resonators are in the dilute regime specified by \eqnref{eq:dilute}. We also assume that $\alpha \neq 0$ is fixed. Then we have the following asymptotics of the capacitance matrix $C_{ij}^\alpha$ as $\epsilon\rightarrow 0$:
		\begin{align}
		C_{11}^\alpha &= \epsilon \mathrm{Cap}_B - \frac{(\epsilon \mathrm{Cap}_B)^2}{4\pi}\sum_{m \neq 0}  \frac{e^{\iu m \alpha L}}{  |mL| } + O(\epsilon^3), \label{eq:c1q}
		\\
		C_{12}^\alpha &= -\frac{(\epsilon \mathrm{Cap}_B)^2}{4\pi}\sum_{m =-\infty}^\infty \frac{e^{\iu m \alpha L} }{  |mL + d| } + O(\epsilon^3). \label{eq:c2q}
		\end{align}
		Taking the imaginary part of \eqnref{eq:c2q}, the corresponding asymptotic formula holds uniformly in $\alpha \in Y^*$.
	\end{lem}

\begin{proof} 
	Recall that the capacitance matrix $C_{ij}^\alpha$ can be written as
	$$
	C_{ij}^\alpha = - \int_{\p D_i} (\mathcal{S}_{D}^{\alpha,0})^{-1}[\chi_{\p D_j}] \dx \sigma,
	$$
	for $\alpha\neq 0$. We shall compute the asymptotics of  $(\mathcal{S}_{D}^{\alpha,0})^{-1}$ for small $\epsilon$.
	
	Let us decompose the Green's function $G^{\alpha,k}$ as
	$$
	G^{\alpha,k}(x,y) = G^k(x,y) + \widetilde{G}^{\alpha,k}(x,y),
	$$
	where
	$$
	\widetilde{G}^{\alpha,k}(x,y):= \sum_{m \neq 0} e^{\iu m \alpha L }  G^k(x,y+(mL,0,0)).
	$$
	Then let us define
	\begin{align*}
	\widetilde{S}_D^{\alpha,0}[\varphi] &=  \int_{\p D} \widetilde{G}^{\alpha,k}(x,y) \varphi(y) \dx \sigma (y).
	\end{align*}
	Note that ${S}_D^{\alpha,0}={S}_D^{0}+\widetilde{S}_D^{\alpha,0}$.
	
	Let us write the quasiperiodic single layer potential $\mathcal{S}_{D}^{\alpha,0}$ in a matrix form as
	\begin{align*}
	\mathcal{S}_{D}^{\alpha,0}
	=
	\begin{bmatrix}
	\mathcal{S}_{D_1}^{\alpha,0}
	& \mathcal{S}_{D_2}^{\alpha,0}|_{\p D_1}
	\\
	\mathcal{S}_{D_1}^{\alpha,0}|_{\p D_2}
	& \mathcal{S}_{D_2}^{\alpha,0}
	\end{bmatrix}, 
	\end{align*}
	and then decompose it as
	\begin{align*}
	\mathcal{S}_{D}^{\alpha,0} &= 
	\begin{bmatrix}
	\mathcal{S}_{D_1}^{0}
	& 0
	\\
	0
	& \mathcal{S}_{D_2}^{0}
	\end{bmatrix}
	+
	\begin{bmatrix}
	\widetilde{S}_{D_1}^{\alpha,0}
	&     {S}_{D_2}^{0}|_{\p D_1} + \widetilde{S}_{D_2}^{\alpha,0}|_{\p D_1}
	\\
	{S}_{D_1}^{0}|_{\p D_2} + \widetilde{S}_{D_1}^{\alpha,0}|_{\p D_2}
	&     \widetilde{S}_{D_2}^{\alpha,0}
	\end{bmatrix}
	\\
	&:= S_I + S_{II}.
	\end{align*}
	
	In order to keep the order of the norms in $L^2(\p D)$ and $H^1(\p D)$ constant as $\epsilon\rightarrow 0$, we let $\L$ and $\H$, respectively, denote the spaces $L^2(\p D)$ and $H^1(\p D)$ along with the inner products
	$$\langle \cdot , \cdot\rangle_{\L} = \frac{1}{|\p D|}\langle \cdot , \cdot\rangle_{L^2(\p D)}, \qquad \langle \cdot , \cdot\rangle_{\H} = \frac{1}{|\p D|}\langle \cdot , \cdot\rangle_{H^1(\p D)}.$$	
	Then, for a fixed $\widetilde\varphi \in L^2(\p B)$, if we define $\varphi\in L^2(\epsilon\p B)$ as $\varphi(x) = \widetilde\varphi(\epsilon^{-1} x)$ we have $\|\varphi\|_{\L} = O(1)$ as $\epsilon \rightarrow 0$.
	
	Next, we estimate the operator norms of $S_I$ and $S_{II}$. We first handle the operator $S_I$. By the scaling property
	$
	\mathcal{S}_{\epsilon B}^0[\varphi] = \epsilon \mathcal{S}_{B}^0[\widetilde\varphi],
	$
	it	can be shown that
	$$
	\big\| \mathcal{S}_{ D_j}^0  \big\|_{\B(\L,\H)} \lesssim \epsilon, \qquad \big\| (\mathcal{S}_{ D_j}^0)^{-1} \big \|_{\B(\H, \L)} \lesssim \epsilon^{-1}, \quad j = 1,2,
	$$
	which implies that
	\begin{equation}\label{SI_estim}
	\big\| S_I\big\|_{\B(\L^2,\H^2)} \lesssim \epsilon,\quad \big\| S_I^{-1}\big\|_{\B(\H^2, \L^2)} \lesssim \epsilon^{-1}.
	\end{equation}
	Here, the notation $A \lesssim B$ means that there exists a constant $K$ independent of $\epsilon$ such that $A \leq K B$ for all small enough $\epsilon$. Further, $\B(X,Y)$ is used to denote the space of bounded linear operators between the normed spaces $(X,\|\cdot\|_X)$ and $(Y,\|\cdot\|_Y)$, and the $\|\cdot\|_{\B(X,Y)}$ norm is defined in the usual way as $\|T\|_{\B(X,Y)}:=\inf\{M:\|T(x)\|_Y\leq M\|x\|_X,\forall x\in X \}$.

	Let us now consider $S_{II}$. We introduce the notation
	\begin{align*}
	z_1 = \left(-\frac{d}{2},0,0\right), \quad z_{1,m} = \left(-\frac{d}{2}+mL,0,0\right), \quad z_2 = \left(\frac{d}{2},0,0\right), \quad z_{2,m} =  \left(\frac{d}{2} +mL,0,0\right).
	\end{align*}
	We have, for small $\epsilon$, that
	\begin{align*}
	\widetilde{\mathcal{S}}_{ D_j}^{\alpha,0} \big|_{\p D_i} [\varphi](x) &= \int_{\p D_j} \sum_{m \neq 0}e^{\iu m\alpha L}\Big(  G^0(x,z_{j,m}) + (y-z_j) \cdot \nabla_y G^0(x,y_m^*) \Big) \varphi(y) \dx \sigma (y)
	\nonumber
	\\
	&=-\sum_{m\neq 0}\frac{e^{\iu m\alpha L}\chi_{\p D_i}(x)}{4\pi |z_i - z_{j,m}	|} \int_{\p D_j} {\varphi}(y) \dx \sigma(y)
	+ O\Big( \sum_{m\neq 0}\frac{\epsilon \int_{\p D_j} |{\varphi}(y)| \dx \sigma(y)}{|z_i-z_{j,m}|^2}\Big). 
	\nonumber
	\end{align*}
	Here, $y_m^*$ means a point on the line segment joining $y$ and $z_j$. Note that the series in the remainder term converges.
	Moreover, the gradient of the remainder term is of the same order. Since $\int_{\p D_j}\varphi \dx\sigma = O(\epsilon^2\| \varphi\|_{\L} )$, we get
	\begin{align}
	\widetilde{\mathcal{S}}_{ D_j}^{\alpha,0} \big|_{\p D_i} [\varphi](x)	&=-\sum_{m\neq 0}\frac{e^{\iu m\alpha L} \chi_{\p D_i}(x)}{4\pi |z_i - z_{j,m}|} \int_{\p D_j} {\varphi}(y) \dx\sigma(y)
	+ O(\epsilon^3 \| \varphi\|_{\L} ),
	\label{eq:single_ij_formula}
	\\
	\nabla_{\p D}\widetilde{\mathcal{S}}_{ D_j}^{\alpha,0} \big|_{\p D_i} [\varphi](x) &= O(\epsilon^3 \| \varphi\|_{\L} ). \nonumber 
	\end{align}
	Here, $\nabla_{\p D}$ is used to denote the surface gradient on $\p D$. From this we have
	$$
	\big\| \widetilde{\mathcal{S}}_{ D_j}^{\alpha,0} \big|_{\p D_i} \big\|_{\B(\L, \H)} \lesssim \epsilon^2.
	$$
	Similarly, we can show that
	\begin{align}
	{\mathcal{S}}_{ D_j}^{0} \big|_{\p D_i} [\varphi](x)  
	&=-\frac{ \chi_{\p D_i}(x)}{4\pi |z_i - z_j
		|} \int_{\p D_j} {\varphi}(y) \dx \sigma(y)
	+ O( \epsilon^3 \| \varphi\|_{\L}),
	\label{eq:single_0_formula}
	\\
	\nabla_{\p D}\mathcal{S}_{ D_j}^{0} \big|_{\p D_i} [\varphi](x) &= O(\epsilon^3 \| \varphi\|_{\L} ),\nonumber
	\end{align}
	and  $\|{\mathcal{S}}_{ D_j}^{0} \big|_{\p D_i}\|_{\B(\L,\H)} \lesssim \epsilon^2$. These imply that
	\begin{equation}\label{SII_estim}
	\| S_{II} \|_{\B(\L^2,\H^2)}  \lesssim \epsilon^2.
	\end{equation}
	
	We now compute the asymptotic behaviour of $(\mathcal{S}_D^{\alpha,0})^{-1}$.
	We use the definition $\phi_j :=  S_I^{-1}[\chi_{\p D_j}]$ and introduce the capacitance of each individual resonator $D_j$ as
	$
	\textrm{Cap}_{D_j} :=\nolinebreak -\int_{\p D_j}\phi_{j} \dx \sigma.
	$
	Note that $\textrm{Cap}_{D_j}  = \epsilon \textrm{Cap}_B$ by \eqref{eq:cap_scale}. Since we know from \eqref{SI_estim} and \eqref{SII_estim} that $S_I^{-1}S_{II} = O(\epsilon)$ in the operator norm, applying the Neumann series gives
	\begin{align}
	(\mathcal{S}_D^{\alpha,0})^{-1} [\chi_{\p D_j}] &= (S_I + S_{II})^{-1} [\chi_{\p D_j}] \nonumber
	\\
	&= (I + S_I^{-1} S_{II})^{-1} S_I^{-1}[\chi_{\p D_j}] \nonumber
	\\
	&= (I - S_I^{-1} S_{II} ) [\phi_j ] + O(\epsilon). \label{SDalpha_inverse}
	\end{align}
	We also have
	$$
	|z_{1}-z_{1,m}| = |mL|, \quad |z_{1} - z_{2,m}| =   |m L + d|.
	$$
	Then, from \eqref{eq:single_ij_formula} and \eqref{eq:single_0_formula}, together with the fact that  $\| \phi_j\|_{\L} = \| \mathcal{S}_{I}^{-1}[\chi_{\p D_j}]\|_{\L}\lesssim\epsilon^{-1}, \ j=1,2$, we obtain the series representations
	\begin{align*}
	S_{II}[\phi_1]|_{\p D_1}
	&=\sum_{m\neq 0}e^{\iu m\alpha L}\frac{\epsilon\textrm{Cap}_{B}}{4\pi |mL|}  \chi_{\p D_{1}}  + O(\epsilon^2),
	\\
	S_{II}[\phi_2]|_{\p D_1}
	&=\sum_{m\in\mathbb{Z}}e^{\iu m\alpha L}\frac{\epsilon\textrm{Cap}_{B}}{4\pi |mL + d|}  \chi_{\p D_{2}}  + O(\epsilon^2).
	\end{align*}

	We are ready to compute the capacitance matrix. From \eqref{SDalpha_inverse}, we have
	\begin{align*}
	C_{11}^\alpha &=- \int_{\p D_1} (\mathcal{S}_D^{\alpha,0})^{-1} [\chi_{\p D_1}] \dx \sigma = - \int_{\p D_1}  (I - S_I^{-1} S_{II} ) [\phi_1]\dx \sigma  - \int_{\p D_1}  O({\epsilon})\dx \sigma \\
	&= - \int_{\p D_1} \phi_1 \dx \sigma + \int_{\p D_1} S_I^{-1}\sum_{m\neq 0}e^{\iu m\alpha L}\frac{\epsilon\textrm{Cap}_{B}}{4\pi |mL|}  \chi_{\p D_{1}} \dx \sigma + O(\epsilon^3)
	\\
	&=\textrm{Cap}_{D_1} -  \sum_{m\neq 0}e^{\iu m\alpha L}\frac{\epsilon\textrm{Cap}_{B}}{4\pi |mL|} \Big(-\int_{\p D_1}   \phi_{1}\dx \sigma\Big) + O(\epsilon^3)	
	\\
	&= \epsilon \textrm{Cap}_{B}   - \sum_{m\neq 0}e^{\iu m\alpha L}\frac{(\epsilon\textrm{Cap}_{B})^2}{4\pi |mL|}   +  O(\epsilon^3).
	\end{align*}
	The expression for $C_{12}^\alpha$ can be derived in the same way.
\end{proof}

	\subsubsection{Band structure and Bloch eigenmodes}
	Define normalized extensions of $V_j^\alpha$ as
	$$S_j^\alpha(x) := \begin{cases} \frac{1}{\sqrt{|D_1|}}\delta_{ij} \quad &x \in D_i, \ i=1,2, \\ \frac{1}{\sqrt{|D_1|}}V_j^\alpha(x) \quad &x \in Y\setminus D, \end{cases}$$ where $|D_1|$ is the volume of one of the resonators ($|D_1|=|D_2|$ thanks to the dimer's symmetry \eqref{dimer_symmetry}). 
	Using similar arguments to those given in \cite{honeycomb, nearzero, highfrequency}, the following two approximation results can be proved.
	
	\begin{thm} \label{thm:char_approx_infinite}
		The characteristic values $\omega_j^\alpha=\omega_j^\alpha(\delta),~j=1,2$, of the operator $\mathcal{A}^{\alpha}(\omega,\delta)$, defined in \eqref{eq:A_quasi_defn}, can be approximated as
		$$ \omega_j^\alpha= \sqrt{\frac{\delta \lambda_j^\alpha }{|D_1|}}  + O(\delta),$$
		where $\lambda_j^\alpha,~j=1,2$, are eigenvalues of the quasiperiodic capacitance matrix $C^\alpha$.
	\end{thm}

	\begin{thm} \label{thm:mode_approx}
		The Bloch eigenmodes $u_j^\alpha,~j=1,2$, corresponding to the resonances $\omega_j^\alpha$, can be approximated as
		$$ u_j^\alpha(x) = a_j S^\alpha_1(x) + b_j S^\alpha_2(x) + O(\delta),$$
		where $\left(\begin{smallmatrix}a_j\\b_j\end{smallmatrix}\right),~j=1,2,$
		 are the eigenvectors of the quasiperiodic capacitance matrix $C^\alpha$, as given by \Cref{lem:evec}.
	\end{thm}
\begin{rmk} \label{rmk:tight_binding}
	From Theorems~\ref{thm:char_approx_infinite} and \ref{thm:mode_approx}, we can see that the capacitance matrix can be considered to be a discrete approximation of the differential problem \eqref{eq:scattering_quasi}, since its eigenpairs directly determine the resonant frequencies and the Bloch eigenmodes (at leading order). This is analogous to the tight-binding model commonly used in the quantum-mechanical SSH system. 
\end{rmk}
\begin{rmk}
	In the quantum-mechanical SSH model, the tight-binding model is typically handled with a \emph{nearest-neighbour approximation}, where only the interactions between neighbouring particles are considered. In this regime, the model is described by the simple Hamiltonian matrix 
	\begin{equation} \label{eq:hamiltonian}
	\begin{bmatrix}
	0 & v+w e^{\iu  L \alpha} \\
	v+w e^{-\iu  L \alpha} & 0
	\end{bmatrix},
	\end{equation}
	where $v$ and $w$ are the two inter-particle coupling constants. Compare this to our discrete approximation, given by the capacitance matrix $C^\alpha$.	If we applied a nearest-neighbour approximation, the capacitance matrix $C^\alpha$ would have the same form as the Hamiltonian \eqref{eq:hamiltonian} (up to an additive diagonal matrix).
	This would be achieved by neglecting the series in \eqref{eq:c1q} and truncating the series in \eqref{eq:c2q}  to $|m|\leq1$ only. However, in classical wave propagation problems such as these, the slow decay of the capacitance matrix means this approximation may be inaccurate. This is because non-negligible interactions exist even between resonators separated by several unit cells. We shall see that this is also the case for finite crystals, in \Cref{sec:finite_nearest_n}.
\end{rmk}

 \begin{rmk} \label{rmk:phase}
 	Theorem \ref{thm:mode_approx} shows that the Bloch eigenmodes are asymptotically constant on each resonator. The value attained on each successive resonator differs by a phase factor determined by $\theta_\alpha$. This theorem was proved in \cite{highfrequency} using layer potential techniques under the assumption $\alpha \neq 0$. By analytic continuation we may extend to any $\alpha \in Y^*$ \cite{phases}.
 \end{rmk}
	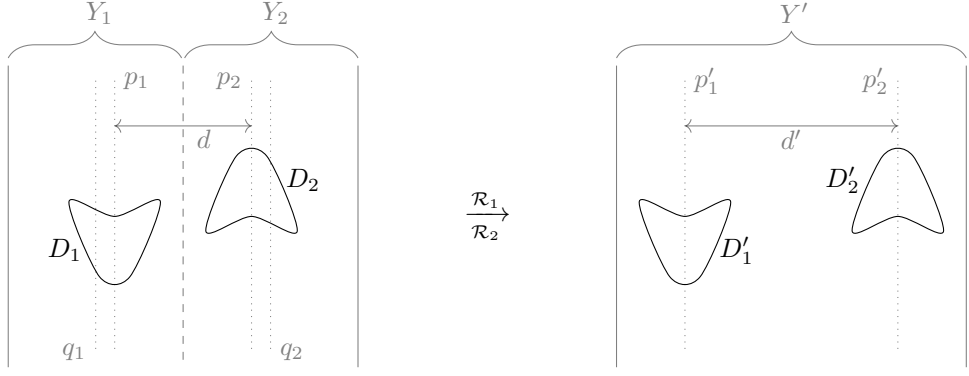
\begin{figure}[tb]
	\centering
	\begin{tikzpicture}[scale=2]
	\begin{scope}
	\draw[opacity=0.5] (-0.5,1) -- (-0.5,-1);
	\draw[opacity=0.5, dashed] (0.65,1) -- (0.65,-1);
	\draw[opacity=0.5, dotted] (1.225,0.9) -- (1.225,-0.9)node[right]{$q_2$};
	\draw[opacity=0.5, dotted] (0.075,0.9) -- (0.075,-0.9)node[left]{$q_1$};
	\draw[opacity=0.5, dotted] (1.1,0.9)  node[left]{$p_2$} -- (1.1,-0.9);
	\draw[opacity=0.5, dotted] (0.2,0.9)  node[right]{$p_1$} -- (0.2,-0.9);
	\draw[opacity=0.5]  (1.8,1) -- (1.8,-1);
	\draw  plot [smooth cycle] coordinates {(-0.1,0.1) (0.2,0) (0.5,0.1) (0.3,-0.4) (0.1,-0.4)} node[xshift=-13pt, yshift=10pt]{$D_1$};
	\draw  plot [smooth cycle] coordinates {(0.8,-0.1) (1.1,0) (1.4,-0.1) (1.2,0.4) (1,0.4)} node[xshift=25pt, yshift=-9pt]{$D_2$};
	\draw[<->, opacity=0.5] (0.2,0.6) -- (1.1,0.6) node[pos=0.65, yshift=-5pt,]{$d$};
	\end{scope}
	\draw (2.65,0) node{$\xrightarrow[\Ro_2]{\Ro_1}$};
	\begin{scope}[xshift=4cm]
	\draw[opacity=0.5] (-0.5,1) -- (-0.5,-1);
	\draw[opacity=0.5]  (1.8,1) -- (1.8,-1);
	\draw  plot [smooth cycle] coordinates {(-0.35,0.1) (-0.05,0) (0.25,0.1) (0.05,-0.4) (-0.15,-0.4)} node[xshift=25pt, yshift=10pt]{$D_1'$};
	\draw  plot [smooth cycle] coordinates {(1.05,-0.1) (1.35,0) (1.65,-0.1) (1.45,0.4) (1.25,0.4)} node[xshift=-15pt, yshift=-9pt]{$D_2'$};
	\draw[<->, opacity=0.5] (-0.05,0.6)  -- (1.35,0.6) node[pos=0.5, yshift=-5pt]{$d'$};
	\draw[opacity=0.5, dotted] (1.35,0.9)  node[left]{$p_2'$} -- (1.35,-0.9);
	\draw[opacity=0.5, dotted] (-0.05,0.9)  node[right]{$p_1'$} -- (-0.05,-0.9);
	\end{scope}
	\begin{scope}[yshift=1.05cm]
	\draw [decorate,opacity=0.5,decoration={brace,amplitude=10pt}]
	(-0.5,0) -- (0.64,0) node [black,midway]{};
	\draw [decorate,opacity=0.5,decoration={brace,amplitude=10pt}]
	(0.66,0) -- (1.8,0) node [black,midway]{};
	\node[opacity=0.5] at (0.1,0.3) {$Y_1$};
	\node[opacity=0.5] at (1.26,0.3) {$Y_2$};
	\end{scope}
	\begin{scope}[xshift=4cm,yshift=1.05cm]
	\draw [decorate,opacity=0.5,decoration={brace,amplitude=10pt}]
	(-0.5,0) -- (1.8,0) node [black,midway]{};
	\node[opacity=0.5] at (0.67,0.3) {$Y'$};	
	\end{scope}
	\end{tikzpicture}
	\caption{Reflections taking $D$ to $D'$.} \label{fig:YY'}
\end{figure}

We now introduce notation which, thanks to the assumed symmetry of the resonators, will allow us to prove topological properties of the chain. Divide $Y$ into two subsets $Y=Y_1\cup Y_2$, where $Y_1 := [-\frac{L}{2},0]\times \R^2$ and let $Y_2 := [0,\frac{L}{2}]\times \R^2$, as depicted in \Cref{fig:YY'}. Define $q_1$ and $q_2$ to be the central planes of $Y_1$ and $Y_2$, that is, the planes $q_1 := \{ -\frac{L}{4}\} \times \R^2$ and $q_2 := \{ \frac{L}{4}\} \times \R^2$. Let $\Ro_1$ and $\Ro_2$ be reflections in the respective planes. Observe that, thanks to the assumed symmetry of each resonator \eqref{resonator_symmetry}, the ``complementary'' dimer $D' = D_1' \cup D_2'$, given by swapping $d$ and $d'$, satisfies $D_i' = \Ro_i D_i$ for $i=1,2$.
Define the operator $T_\alpha$ on the set of $\alpha$-quasiperiodic functions $f$ on $Y$ as
$$T_\alpha f(x) := \begin{cases} e^{-\iu \alpha L}\overline{f(\Ro_1x)}, \quad &x\in Y_1, \\ \overline{f(\Ro_2x)}, &x\in Y_2, \end{cases}$$
where the factor $e^{-\iu \alpha L}$ is chosen so that the image of a continuous ($\alpha$-quasiperiodic) function is continuous.

We will now proceed to use $T_\alpha$ to analyse the different topological properties of the two dimer configurations. Define the quantity ${C_{12}^{\alpha}}'$ analogously to $C_{12}^\alpha$ but on the dimer $D'$, that is, to be the top-right element of the corresponding quasiperiodic capacitance matrix, defined in \eqref{eq:qp_capacitance}.

\begin{lem}\label{lem:cc'}
	We have
	\begin{equation*}\label{eq:cc'}
		{C_{12}^{\alpha}}' = e^{-\iu \alpha L} \overline{C_{12}^\alpha}.\end{equation*}
	Consequently, if $d = d' = \frac{L}{2}$ then $C_{12}^{\pi/L} = 0$. 
\end{lem}
\begin{proof}
	Define ${V_1^\alpha }', {V_2^\alpha }'$ by \eqnref{eq:V_quasi} but on $D'$ instead of $D$. Observe that $T_\alpha {V_1^\alpha}' = e^{-\iu \alpha L}\overline{V_1^\alpha }$ and $T_\alpha {V_2^\alpha}' = \overline{V_1^\alpha }$. 
	Then, we find that
	\begin{align*}
	{C_{12}^{\alpha}}' &=  \int_{Y\setminus D'} \overline{\nabla {V_1^\alpha}'} \cdot \nabla {V_2^\alpha}'\dx x \\
	&= \int_{Y\setminus D} \overline{\nabla {T_\alpha V_1^\alpha}'} \cdot \nabla {T_\alpha V_2^\alpha}'\dx x \\
	& = e^{-\iu \alpha L}\int_{Y\setminus D} \nabla {V_1^\alpha} \cdot \overline{\nabla {V_2^\alpha}}\dx x \\
	&= e^{-\iu \alpha L} \overline{C_{12}^\alpha}.
	\end{align*}
	At $\alpha = \pi/L$, we have ${C_{12}^{\pi/L}}' = -C_{12}^{\pi/L}$. Moreover, if $d=d'$, the symmetry of the structure means that ${C_{12}^{\pi/L}}' = C_{12}^{\pi/L}$ so it must be the case that $C_{12}^{\pi/L} = 0$.
\end{proof}

\begin{lem}\label{lem:c=0}
	We assume that $D$ is in the dilute regime specified by \eqnref{eq:dilute}. Then, for $\epsilon$ small enough,
	\begin{itemize}
		\item[(i)] $\mathrm{Im}\ C_{12}^\alpha > 0$ for $0<\alpha<\pi/L$ and $\mathrm{Im}\ C_{12}^\alpha < 0$ for $-\pi/L<\alpha<0$. In particular, $\mathrm{Im}\ {C_{12}^{\alpha}}$ is zero if and only if $\alpha \in\{ 0, \pi/L \}$.
		\item [(ii)] $C_{12}^{\alpha}$ is zero if and only if both $d = d'$ and $\alpha = \pi/L$.
		\item [(iii)] $C_{12}^{\pi/L} < 0$  when $d<d'$ and $C_{12}^{\pi/L} > 0$ when $d>d'$. In both cases we have $C_{12}^{0} < 0$.
	\end{itemize}
\end{lem}
The proof of \Cref{lem:c=0} is given in Appendix \ref{app:lemma}. This lemma describes the crucial properties of the behaviour of the curve $\{C_{12}^\alpha:\alpha\in Y^*\}$ in the complex plane. The periodic nature of $Y^*$ means that this is a closed curve. Part (i) tells us that this curve crosses the real axis in precisely two points. Taken together with (iii), we know that this curve winds around the origin in the case $d>d'$, but not in the case $d<d'$. 

\begin{thm} \label{thm:band_gap}
 If $d\neq d'$ there exists a band gap, for $\alpha$ away from zero. That is, for any small $\alpha_0>0$,  we have that
	$$\max_{|\alpha|>\alpha_0} \omega_1^\alpha < \min_{|\alpha|>\alpha_0} \omega_2^\alpha,$$
	for small enough $\epsilon$ and $\delta$.
\end{thm}

	The proof of \Cref{thm:band_gap} is given in Appendix~\ref{app:lemma_gap}. The argument is based on representing the first and second resonant frequencies as
\begin{align*}
\omega_1^\alpha = \sqrt{\frac{\delta\left(C_{11}^\alpha -|C_{12}^\alpha|\right)}{|D_1|}} + O(\delta), \qquad     \omega_2^\alpha = \sqrt{\frac{\delta\left(C_{11}^\alpha +|C_{12}^\alpha|\right)}{|D_1|}} + O(\delta),
\end{align*}
and making use of the fact that, in the dilute regime and for fixed $\alpha\neq0$, the capacitance coefficients can be expanded using \Cref{lem:cap_estim_quasi}.

\begin{rmk}
	Part (ii) of \Cref{lem:c=0} is a particularly deep result which shows that the dilute crystal has a degeneracy precisely when $d=d'$. The methods developed in \cite{honeycomb} can be applied to show that the dispersion relation has a Dirac cone at $\alpha = \pi/L$ in this case. As $d$ increases across the point where $d=d'$, the band gap closes (to form a Dirac cone) before reopening. In \Cref{thm:phase} we will show that the reopened band gap has a non-trivial topology, similar to what has been observed in other systems (for example, in \cite{Yves1}).
\end{rmk}

Combining the results of \Cref{lem:c=0}, \Cref{lem:evec} and \Cref{thm:mode_approx}, we obtain the following result concerning the band inversion that takes place between the two geometric regimes $d<d'$ and $d>d'$.
\begin{prop} \label{prop:bandinv}
	For $\epsilon$ small enough, the band structure at $\alpha = \pi/L$ is inverted between the $d<d'$ and $d>d'$ regimes. In other words, the eigenfunctions associated with the first and second bands at $\alpha = \pi/L$ are given, respectively, by
	\begin{equation*}
	u_1^{\pi/L}(x) = S_1^{\pi/L}(x)+S_2^{\pi/L}(x)+O(\delta), \quad u_2^{\pi/L}(x) = S_1^{\pi/L}(x)-S_2^{\pi/L}(x)+O(\delta), 
	\end{equation*}
	when $d<d'$ and by
	\begin{equation*}
	u_1^{\pi/L}(x) = S_1^{\pi/L}(x)-S_2^{\pi/L}(x)+O(\delta), \quad u_2^{\pi/L}(x) = S_1^{\pi/L}(x)+S_2^{\pi/L}(x)+O(\delta), 
	\end{equation*}
	when $d>d'$.
\end{prop}
The eigenmode $S_1^{\pi/L}(x)+S_2^{\pi/L}(x)$ is constant and attains the same value on both resonators, while the eigenmode $S_1^{\pi/L}(x)-S_2^{\pi/L}(x)$ has values of opposite sign on the two resonators. They therefore correspond, respectively, to monopole and dipole modes, and \Cref{prop:bandinv} shows that the monopole/dipole nature of the first two Bloch eigenmodes are swapped between the two regimes. We will now proceed to define a topological invariant which we will use to characterise the topology of a chain and prove how its value depends on the relative sizes of $d$ and $d'$. This invariant is intimately connected with the band inversion phenomenon, and we will prove that it is non-trivial only if $d>d'$.

\begin{thm} \label{thm:phase}
	We assume that $D$ is in the dilute regime specified by \eqnref{eq:dilute}. Then the Zak phase $\varphi_j^z, j = 1,2$, defined by 
	$$\varphi_j^z := \iu \int_{Y^*} \big\langle u_j^\alpha, \frac{\p }{\p \alpha} u_j^\alpha \big\rangle \dx \alpha,$$ 
	where $\langle \cdot, \cdot \rangle$ denotes the $L^2(D)$-inner product, satisfies
	$$ \varphi_j^z = \begin{cases}
	0, \quad &\text{if} \ \ d < d', \\
	\pi, \quad &\text{if} \ \ d > d',
	\end{cases}$$
	for $\epsilon$ and $\delta$ small enough.
\end{thm}
\begin{proof}
	We compute the Zak phase $\phi^z_j, j=1,2,$ of the first and second band, respectively. Observe that
	$$  \langle S_1^\alpha,  S_1^\alpha \rangle = 1, \qquad  \langle S_2^\alpha,  S_2^\alpha \rangle = 1, \qquad \langle S_1^\alpha,  S_2^\alpha \rangle = 0,$$
	and in $D$ we have
	$$ \frac{\p}{\p \alpha}  S_1^\alpha \equiv 0, \qquad \frac{\p}{\p \alpha}  S_2^\alpha \equiv 0,$$
	for all $\alpha \in Y^*$. By Theorem \ref{thm:mode_approx} and Lemma \ref{lem:evec}, it follows that 
	\begin{align*}
	\langle u_j^\alpha, \frac{\p }{\p \alpha} u_j^\alpha \rangle  = \frac{\iu}{2}\frac{\p \theta_\alpha }{\p \alpha} + O(\delta),
	\end{align*}
	so the Zak phase is given by
	$$
	\varphi_j^z = -\frac{1}{2}\left[\theta_\alpha \right]_{Y^*} + O(\delta).
	$$
	Since we know that $\varphi_j^z$ is an integer multiple of $\pi$, we have for small enough $\delta$ that
	\begin{equation} \label{eq:theta}
	\varphi_j^z = -\frac{1}{2}\left[\theta_\alpha \right]_{Y^*}.
	\end{equation}
	This representation of $\varphi_j^z$, which is analogous to well-known results for Hamiltonian systems such as the SSH model \cite{shortcourse}, shows that the Zak phase is given by the change in the argument of $C_{12}^\alpha$ as $\alpha$ varies over the Brillouin zone $Y^*$. We can see from parts (i) and (iii) of \Cref{lem:c=0} that the winding number of the origin depends on whether $d<d'$ or $d>d'$. In the two cases we have, respectively, $\left[\theta_\alpha \right]_{Y^*} = 0$ and $\left[\theta_\alpha \right]_{Y^*} = -2\pi$. Therefore, if $\delta$ is small enough, we have that
	\begin{equation*}\varphi_j^z = \begin{cases}
	0, \quad &d < d', \\
	\pi, \quad &d' < d.
	\end{cases}\tag*{\qedhere}\end{equation*}
\end{proof}
 
\begin{rmk} \label{rmk:dilute}
	The dilute assumption is not necessary to conclude that the Zak phase is non-trivial for certain configurations. Combining \Cref{lem:cc'} and \eqref{eq:theta}, both of which are valid without assumptions of diluteness, we find that ${\varphi_j^z}' - \varphi_j^z = \pi$, where ${\varphi_j^z}'$ is the Zak phase of the crystal with resonator separation $d'$ instead of $d$ (as in \Cref{fig:YY'}). The assumption of diluteness is invoked to prove part (ii) of \Cref{lem:c=0}, which shows that there are only two different topological regimes and that degeneracy occurs only at $d=d'$. We conjecture that this is true even without the dilute assumption, in which case it is not hard to prove \Cref{prop:bandinv} and \Cref{thm:phase}.
\end{rmk}

Theorem \ref{thm:phase} shows that the Zak phase of the crystal is non-zero precisely when $d > d'$. The bulk-boundary correspondence suggests that we can create topologically protected subwavelength edge modes by joining half-space subwavelength crystals, one with $\varphi_j^z = 0$ and the other with $\varphi_j^z = \pi$. According to \Cref{rmk:dilute}, this is also valid in the non-dilute case. In \Cref{sec:finite_crystals}, we will study a finite chain that is designed with this principle in mind and demonstrate that it exhibits an edge mode that is stable with respect to symmetry-preserving imperfections. 

\begin{figure*} [tbh]
	\begin{center}
		\includegraphics[height=5.0cm]{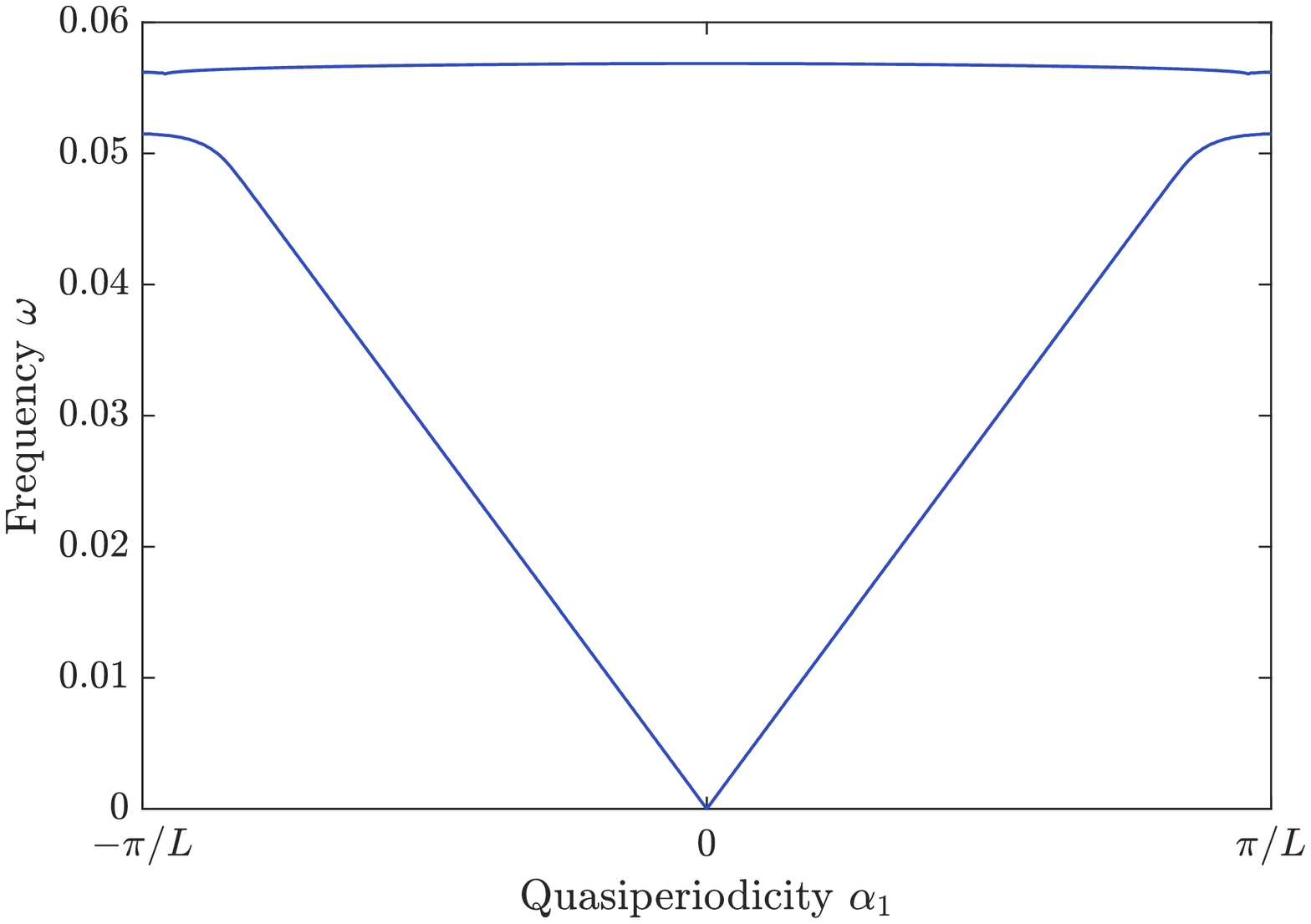}
		\hspace{0.1cm}
		\includegraphics[height=5.0cm]{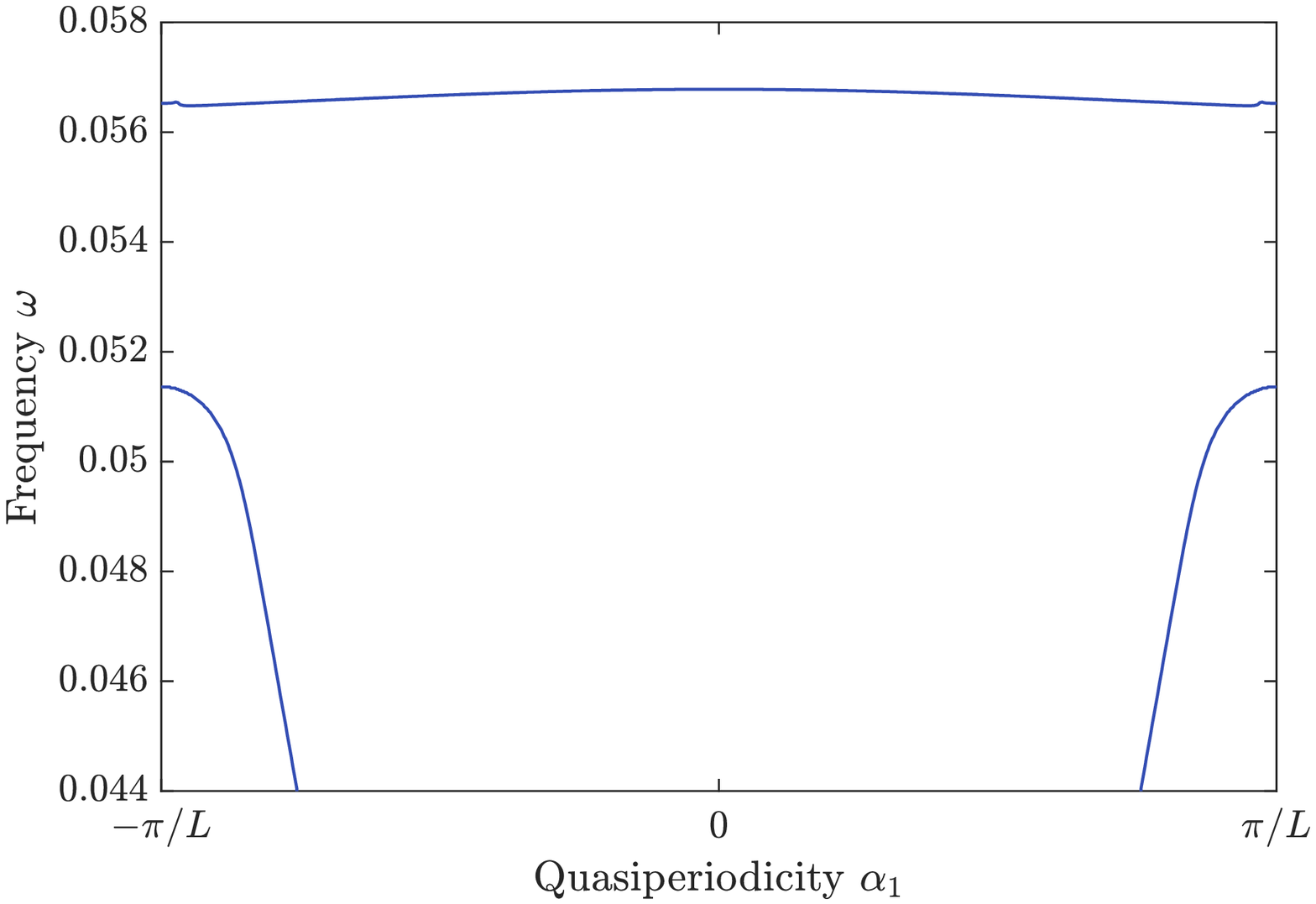}
		\caption{ (Dilute case)
			The full first two bands (left) and magnification of the band gap region (right). The resonator radius was $R=1$ with separations $d=12$ and $d'=42$, corresponding to a dilute crystal.} 
		\label{fig:bandgap_dilute}
		\includegraphics[height=5.0cm]{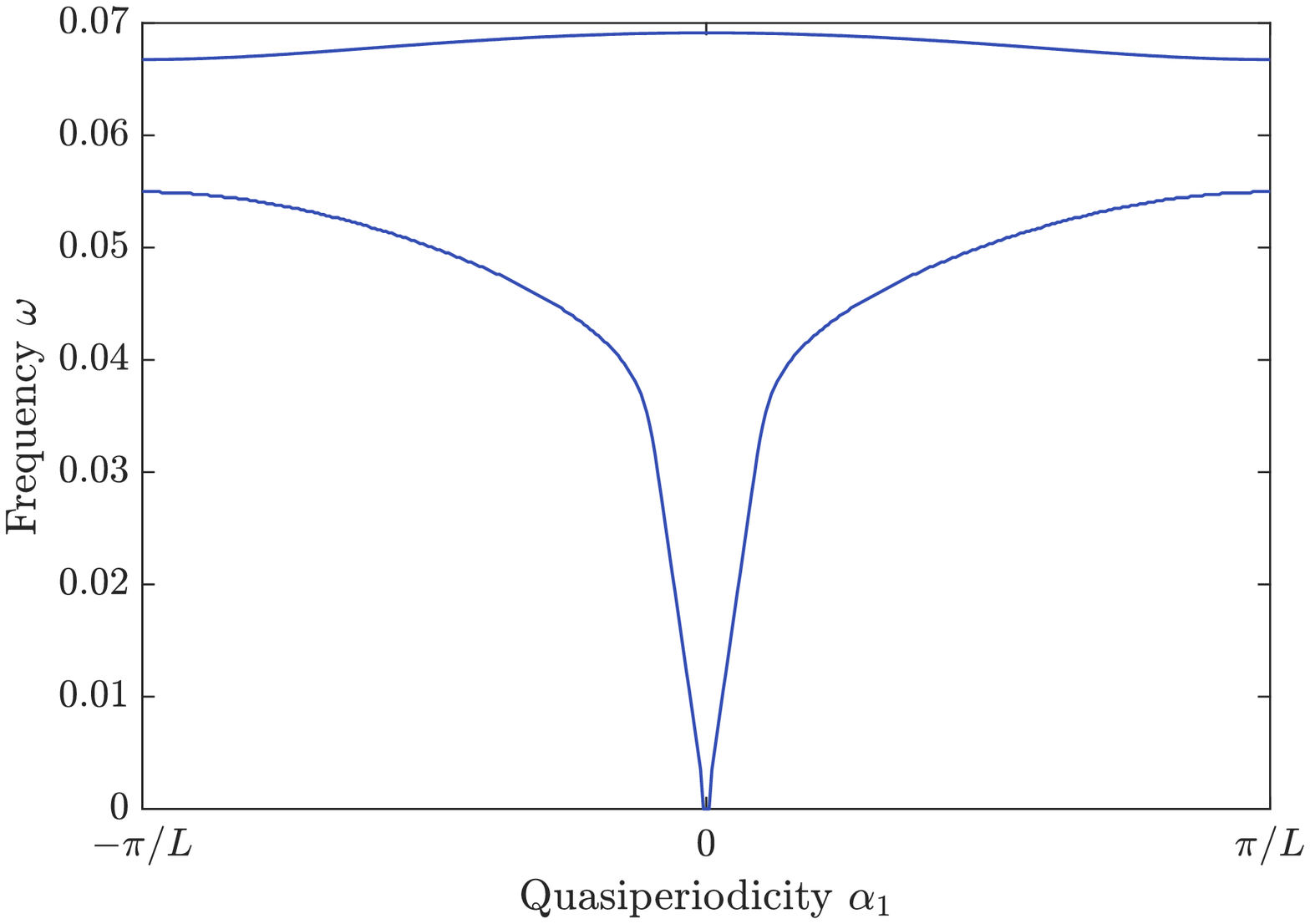}
		\hspace{0.1cm}
		\includegraphics[height=5.0cm]{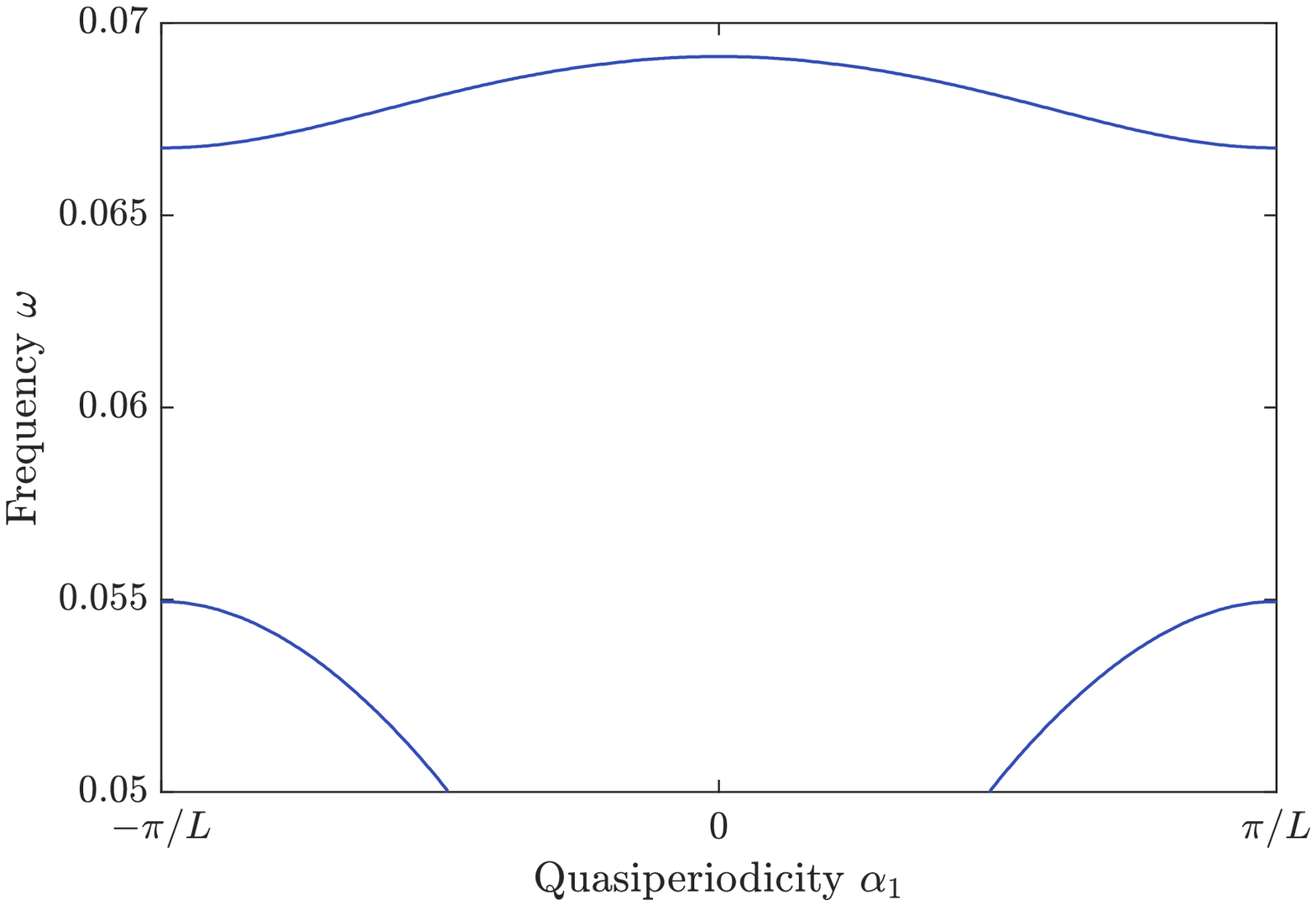}
		\caption{ (Non-dilute case)
			The full first two bands (left) and magnification of the band gap region (right). The resonator radius was $R=1$ with separations $d=3$ and $d' = 6$, corresponding to a non-dilute crystal. } 
		\label{fig:bandgap_ndilute}
	\end{center}
\end{figure*}

\subsection{Numerical computations} \label{sec:inf_numerics}
The band structure and the Bloch eigenmodes were computed using the multipole expansion method derived in Appendix \ref{app:multipole}. This relies on the assumption that the resonators are spherical, which is a special case of the more general geometry considered above. As shown in Theorem \ref{thm:mode_approx}, the Bloch eigenmodes are asymptotically constant on each domain $D_i$ and hence accurate and efficient computations can be achieved by approximating functions by only the first term in the multipole expansion.

All the numerical computations in this paper were performed for the example of acoustic waves being scattered by air bubbles in water. This is a classic example of subwavelength resonance, where the resonant frequency of a single bubble is known as the \emph{Minnaert resonance} \cite{first,minnaert}. Throughout this paper, we use $\delta = 10^{-3}$, which is roughly the density contrast between air and water. We also use the material parameters $d=12, d' = 42, L=54$ to exemplify a dilute crystal, and parameters $d=3, d'=6, L=9$ to exemplify a non-dilute crystal. 

\subsubsection{Band structure}
Figures \ref{fig:bandgap_dilute} and \ref{fig:bandgap_ndilute} show the band structure in a dilute and a non-dilute crystal, respectively. In the dilute case, the sound cone occupies a majority of the Brillouin zone, which is expected due to the lower interactions between the resonators. Both crystals show the opening of a band gap.

\subsubsection{Bloch eigenmodes}
Figure \ref{fig:mode} shows the first two Bloch eigenmodes for the crystal at $\alpha = \pi/L$ in the cases $d<d'$ and $d>d'$. The band inversion is clearly seen: when $d>d'$ the monopole/dipole modes correspond to the second/first mode, respectively.  The band inversion property demonstrates the fact that the crystal has a non-zero Zak phase when $d>d'$. As $\alpha$ varies, the phase shift $\theta_\alpha$ between the values of the eigenmodes winds around 0, resulting in band inversion at some point $\alpha\in Y^*$.

\begin{figure*} [tbh]
	\begin{center}
		\begin{subfigure}[b]{0.2\linewidth} 
		\centering
		\includegraphics[trim={0.65cm 0 2.6cm 0},clip,height=4.5cm]{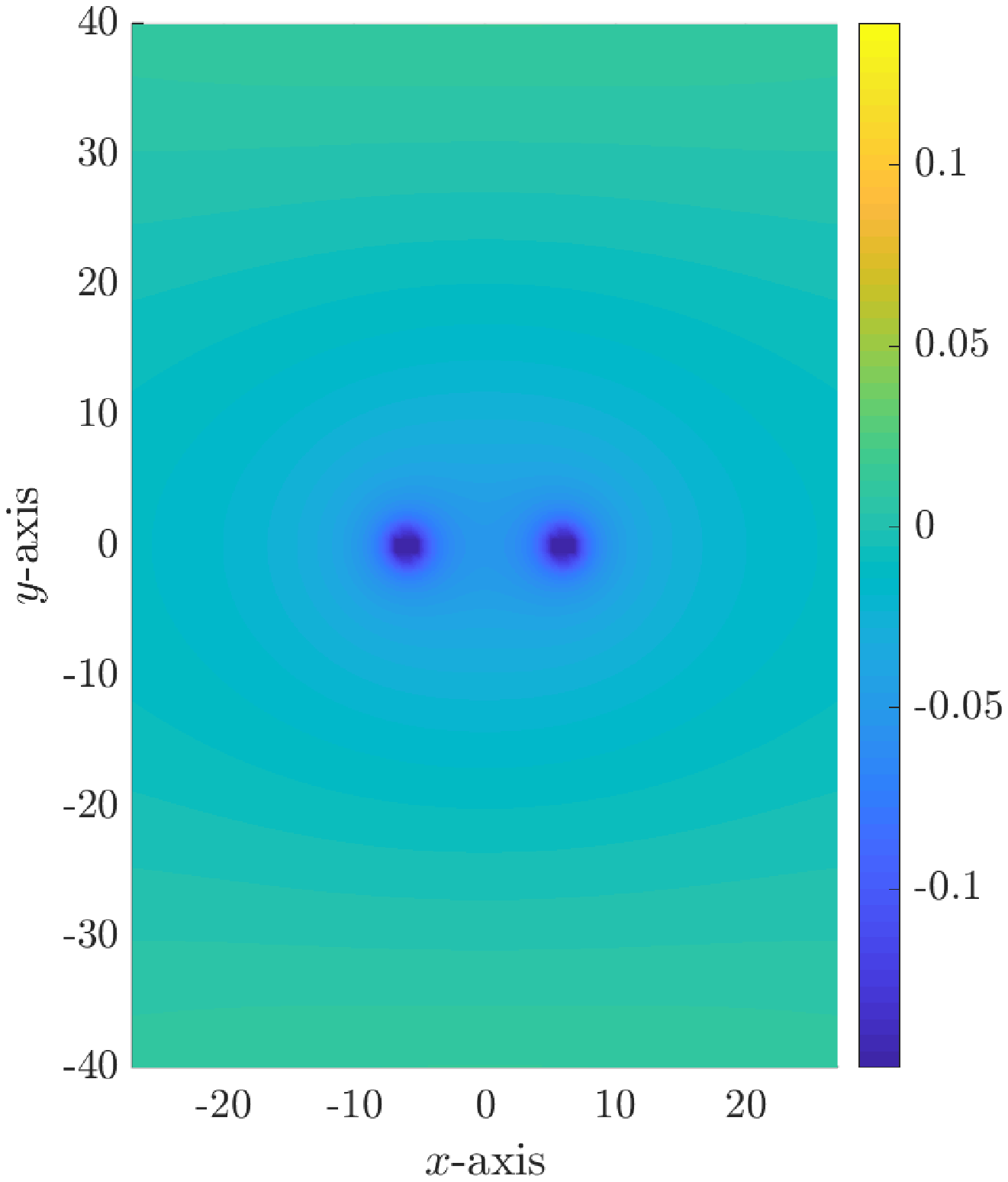}
		\caption{Case $d<d'$: 1\textsuperscript{st} eigenmode.}
		\end{subfigure}
		\hspace{0.6cm}
		\begin{subfigure}[b]{0.2\linewidth}
		\centering
		\includegraphics[trim={0.65cm  0 2.6cm 0},clip,height=4.5cm]{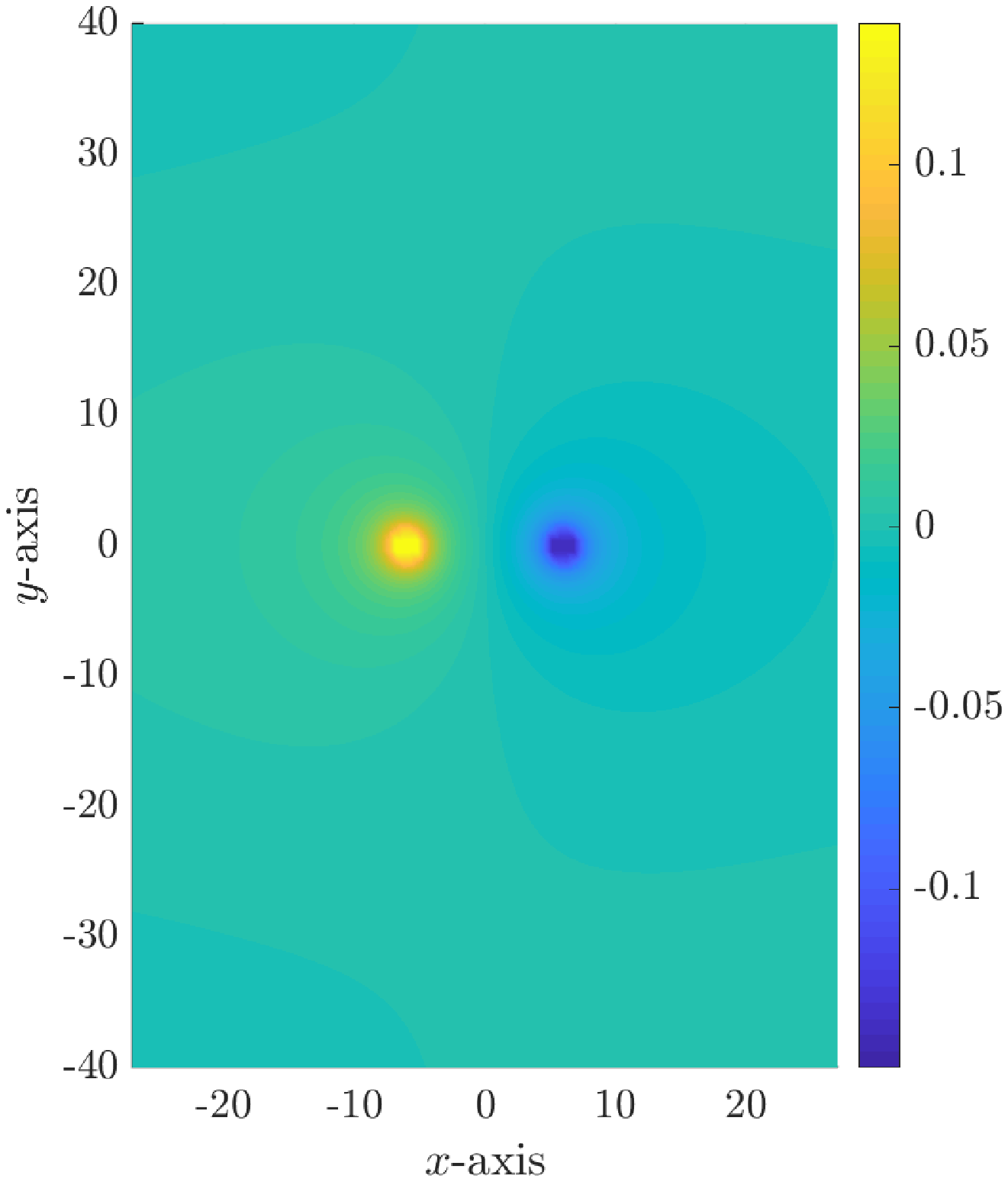}
		\caption{Case $d<d'$: 2\textsuperscript{nd} eigenmode.}
		\end{subfigure}
	\hspace{0.8cm}
		\begin{subfigure}[b]{0.2\linewidth}\centering
		\includegraphics[trim={0.65cm 0 2.6cm 0},clip,height=4.5cm]{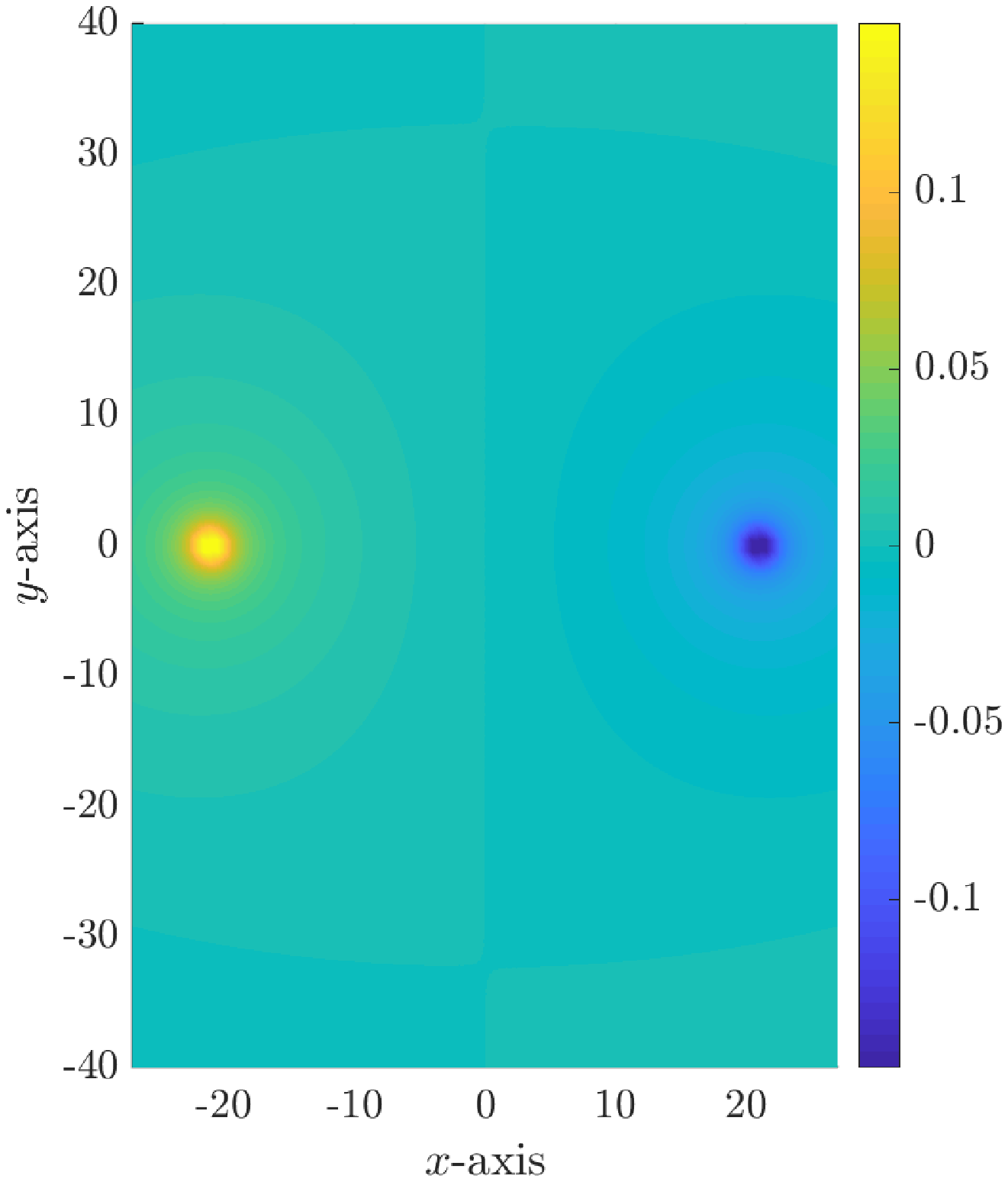}
		\caption{Case $d>d'$: 1\textsuperscript{st} eigenmode.}
		\end{subfigure}
	\hspace{0.6cm}
		\begin{subfigure}[b]{0.2\linewidth}\centering
		\includegraphics[trim={0.65cm 0 0 0},clip,height=4.5cm]{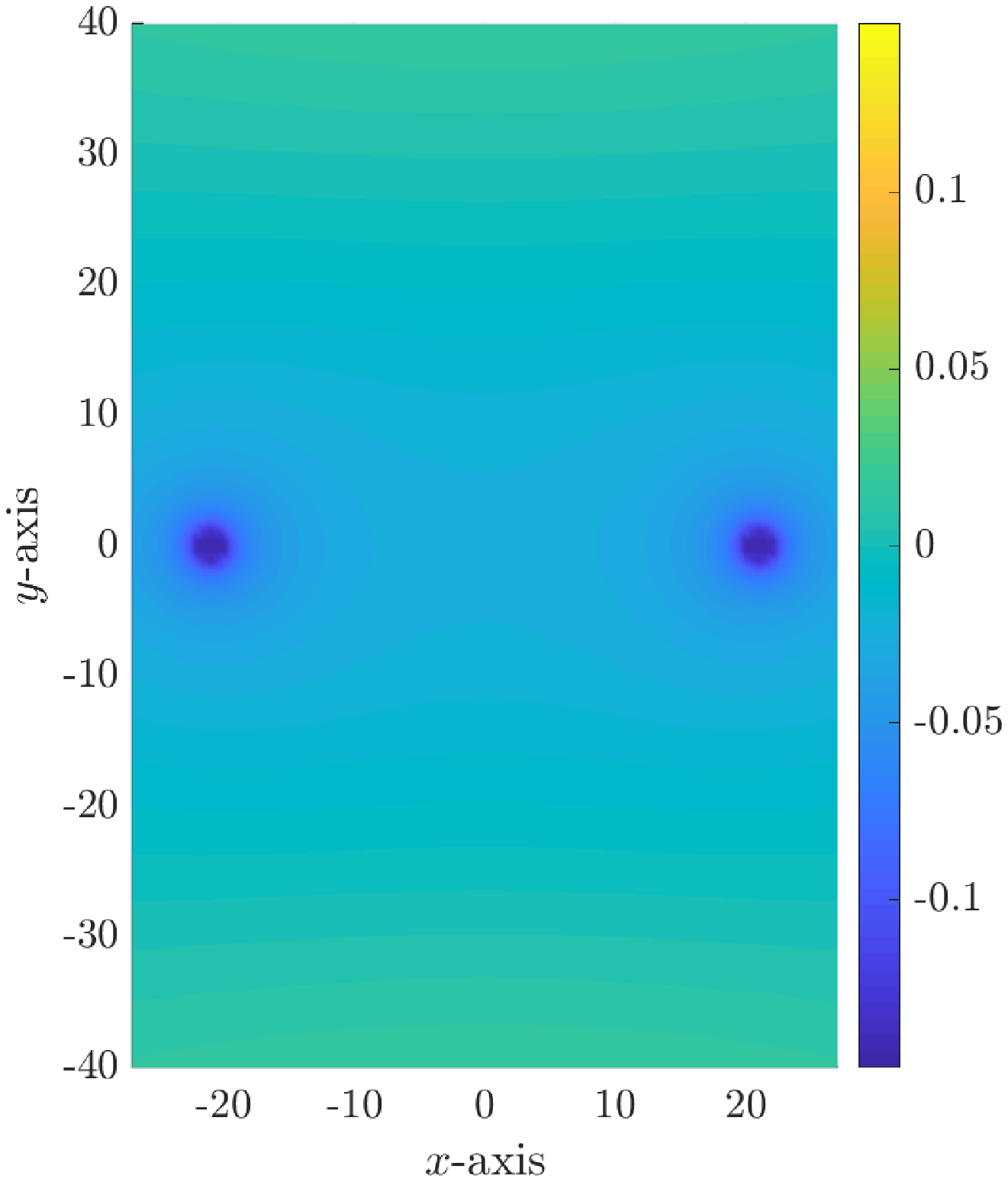}
		\caption{Case $d>d'$: 2\textsuperscript{nd} eigenmode.}
		\end{subfigure}
		\caption{The first and second Bloch eigenmodes at $\alpha = \pi/L$ for the two cases $d<d'$ and $d>d'$ in the dilute regime. The case $d>d'$ exhibits band inversion: the first eigenmode behaves as a dipole while the second mode behaves as a monopole.} 
		\label{fig:mode}
	\end{center}
\end{figure*}

\section{Finite chains of subwavelength resonators} \label{sec:finite_crystals}
In this section, we will study a finite chain of resonators which has been carefully designed to support topologically protected edge modes. Specifically, we assume that $D$ has the form
\begin{equation} \label{finite_form}
D = \left(\bigcup_{n=-M}^{M} D_0 + n(d+d',0,0)  \right) \bigcup \left( \bigcup_{n=-M+1}^M D_0 + n(d+d',0,0) - (d',0,0)\right),
\end{equation}
where $D_0$ is a single repeating resonator. In other words, $D$ consists of an odd number  $N$ of identical resonators ($N = 4M+1$) with alternating distances $d$ and $d'$ that are swapped at the middle resonator. An example of such a configuration is depicted in \Cref{fig:finite}. This structure is based on the intuition that if one joins together two chains with different topological properties, a protected edge mode will occur at the interface (this is the principle of bulk-boundary correspondence). In \Cref{fig:finite} it is shown how on either side of the central resonator (which constitutes the ``edge'') one can associate each successive pair of resonators with dimers belonging to infinite chains that have different Zak phases.


	\begin{figure}[t]
	\centering
	\begin{tikzpicture}[scale=1.1]
	\begin{scope}
	\draw (0.65,0) coordinate (start1) circle (8pt);
	\draw[<->, opacity=0.5] (0.65,0) -- (2.05,0) node[pos=0.5, yshift=-7pt,]{$d'$};
	\draw[<->, opacity=0.5] (-0.75,0) -- (0.65,0) node[pos=0.5, yshift=-7pt,]{$d'$};
	\end{scope}
	
	\begin{scope}[xshift=-1.85cm]
	\draw (0.2,0) circle (8pt);
	\draw[<->, opacity=0.5] (0.2,0) -- (1.1,0) node[pos=0.5, yshift=-7pt,]{$d$};
	\begin{scope}[xshift = 1.3cm]
	\draw (-0.2,0) circle (8pt);
	\end{scope};
	\end{scope};
	
	\begin{scope}[xshift=1.85cm]
	\draw[<->, opacity=0.5] (0.2,0) -- (1.1,0) node[pos=0.5, yshift=-7pt,]{$d$};
	\draw (0.2,0) circle (8pt);		
	\begin{scope}[xshift = 1.3cm]
	\draw (-0.2,0) circle (8pt);
	\end{scope};
	\end{scope};
	
	\begin{scope}[xshift=4.05cm]
	\draw (0.2,0) circle (8pt);	
	\begin{scope}[xshift = 1.3cm]
	\draw (-0.2,0) circle (8pt);
	\end{scope};
	\end{scope};
	
	\begin{scope}[xshift=6.25cm]
	\draw (0.2,0) circle (8pt);	
	\begin{scope}[xshift = 1.3cm]
	\draw (-0.2,0) circle (8pt);
	\end{scope};
	\end{scope};
	
	\begin{scope}[xshift=-4.05cm]
	\draw (0.2,0) circle (8pt);
	\begin{scope}[xshift = 1.3cm]
	\draw (-0.2,0) circle (8pt);
	\end{scope};
	\end{scope};	
	
	\begin{scope}[xshift=-6.25cm]
	\draw (0.2,0) circle (8pt);
	\begin{scope}[xshift = 1.3cm]
	\draw (-0.2,0) circle (8pt);
	\end{scope};
	\end{scope};
		
	
	\begin{scope}[xshift=-4.05cm]
	\draw [decorate,opacity=0.5,decoration={brace,amplitude=10pt},xshift=-4pt,yshift=0pt]
	(-0.25,0.5) -- (1.8,0.5) node [black,midway,xshift=-0.6cm]{};	
	\end{scope}
	
	\begin{scope}[xshift=-1.85cm]
	\draw [decorate,opacity=0.5,decoration={brace,amplitude=10pt},xshift=-4pt,yshift=0pt]
	(-0.25,0.5) -- (1.8,0.5) node [black,midway,xshift=-0.6cm]{};	
	\end{scope}
	
	\begin{scope}[xshift=0.72cm]
	\draw [decorate,opacity=0.5,decoration={brace,amplitude=10pt},xshift=-4pt,yshift=0pt]
	(-0.25,0.5) -- (1.8,0.5) node [black,midway,xshift=-0.6cm]{};	
	\end{scope}
	
	\begin{scope}[xshift=2.92cm]
	\draw [decorate,opacity=0.5,decoration={brace,amplitude=10pt},xshift=-4pt,yshift=0pt]
	(-0.25,0.5) -- (1.8,0.5) node [black,midway,xshift=-0.6cm]{};	
	\end{scope}
	
	
	\begin{scope}[xshift=-4.05cm-4pt]
	\draw[opacity=0.5]    (0.775,0.9) to[out=90,in=-100] (1.54,1.5);
	\end{scope}
	
	\begin{scope}[xshift=-1.85cm-4pt]
	\draw[opacity=0.5]    (0.775,0.9) to[out=100,in=-80] (-0.64,1.5);
	\end{scope}
	

	\begin{scope}[xshift=0.72cm-4pt]
	\draw[opacity=0.5]    (0.775,0.9) to[out=90,in=-100] (2.04,1.5);
	\end{scope}

	\begin{scope}[xshift=2.92cm-4pt]
	\draw[opacity=0.5]    (0.775,0.9) to[out=100,in=-80] (-0.14,1.5);
	\end{scope}
	
	\begin{scope}[xshift=-2.4cm]
	\draw [decorate,opacity=0.5,decoration={brace,amplitude=10pt,mirror},xshift=-4pt,yshift=0pt]
	(-1.125,1.9) -- (0.925,1.9) node [black,midway,xshift=-0.6cm]{};
	\begin{scope}[xshift=-4pt]
	\draw[opacity=0.5, dotted]    (-1.25,2) -- (-1.25,2.7);
	\draw[opacity=0.5, dotted]    (1.05,2) -- (1.05,2.7);
	\end{scope}
	\begin{scope}[xshift = -4pt-0.75cm]
	\draw[opacity=0.5] (0.2,2.3) circle (8pt);	
	\node[opacity=0.5] at (0.65,2.9) {$\varphi_j^z = 0$};
	\begin{scope}[xshift = 1.3cm]
	\draw[opacity=0.5] (-0.2,2.3) circle (8pt);
	\end{scope}
	\end{scope}
	\end{scope}
	
	\begin{scope}[xshift=2.87cm]
	\draw [decorate,opacity=0.5,decoration={brace,amplitude=10pt,mirror},xshift=-4pt,yshift=0pt]
	(-1.125,1.9) -- (0.925,1.9) node [black,midway,xshift=-0.6cm]{};
	\begin{scope}[xshift=-4pt]
	\draw[opacity=0.5, dotted]    (-1.25,2) -- (-1.25,2.7);
	\draw[opacity=0.5, dotted]    (1.05,2) -- (1.05,2.7);
	\end{scope}
	\begin{scope}[xshift = -4pt-0.75cm]
	\draw[opacity=0.5] (-0.05,2.3) circle (8pt);	
	\node[opacity=0.5] at (0.65,2.9) {$\varphi_j^z = \pi$};
	\begin{scope}[xshift = 1.3cm]
	\draw[opacity=0.5] (0.05,2.3) circle (8pt);
	\end{scope}
	\end{scope}
	\end{scope}
	\end{tikzpicture}
	\caption{Two-dimensional cross-section of a finite dimer chain with 13 resonators, heuristically showing how to identify unit cells with different Zak phases on either side of the edge.} \label{fig:finite}
\end{figure}
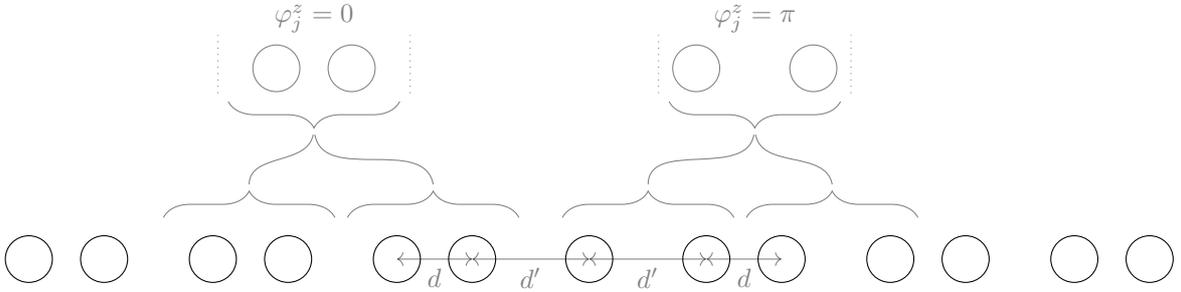

We model wave propagation in the crystal $D$ by the Helmholtz problem

\begin{equation} \label{eq:scattering_finite}
\left\{
\begin{array} {ll}
\ds \Delta {u}+ \omega^2 {u}  = 0 \quad &\text{in } \R^3 \setminus \p D, \\
\nm
\ds  {u}|_{+} -{u}|_{-}  =0  \quad &\text{on } \partial D, \\
\nm
 \ds  \delta \frac{\partial {u}}{\partial \nu} \bigg|_{+} - \frac{\partial {u}}{\partial \nu} \bigg|_{-} =0 \quad& \text{on } \partial D, \\
\nm
 \ds |x| \left(\tfrac{\p}{\p|x|}-\iu \omega\right)u \to 0 
 &\text{as } {|x|} \rightarrow \infty.
\end{array}
\right.
\end{equation}

\subsection{Integral equation formulation of the problem}
The solution $u$ of \eqref{eq:scattering_finite} can be represented as
\begin{equation*} \label{eq:helm-solution}
u =\mathcal{S}_{D}^{\omega} [\Psi],
\end{equation*}
for some density $\Psi \in  L^2(\p D)$. Then, analogous to the approach used in the quasiperiodic case in \Cref{sec:formulation_quasip}, the jump relations can be used to show that~\eqref{eq:scattering_finite} is equivalent to the boundary integral equation
\begin{equation}  \label{eq:boundary}
\mathcal{A}(\omega, \delta)[\Psi] =0,  
\end{equation}
where
$$
\mathcal{A}(\omega, \delta) := -\lambda I + \mathcal{K}_D^{\omega,*}, \quad \lambda := \frac{1+\delta}{2(1-\delta)}.
$$

\subsection{Capacitance matrix}\label{subsec:cap}
Similar to the quasiperiodic case in \Cref{subsec:cap_quasi}, the resonant frequencies and eigenmodes of the finite chain can be expressed in terms of the capacitance matrix. Let $V_j, j=1,...,N$ be the solution to
\begin{equation} \label{eq:V}
\begin{cases}
\ds \Delta V_j =0 \quad &\mbox{in}~R^3\setminus D,\\
\ds V_j = \delta_{ij} &\mbox{on}~\partial D_i,\\
\ds V_j(x) = O\left(\tfrac{1}{|x|}\right) & \text{as } |x|\rightarrow \infty.
\end{cases}
\end{equation}
We define the capacitance coefficients matrix $C=(C_{ij})$ by
\begin{equation} \label{eq:finite_capacitance}
C_{ij} := \int_{\R^3\setminus D}\nabla V_i \cdot\nabla V_j  \dx x,\quad i,j=1,...,N.
\end{equation}
Once again, we can use the jump conditions to show that the capacitance coefficients $C_{ij}$ are also given by
$$ C_{ij} = - \int_{\partial D_i} \psi_j \dx\sigma,\quad i,j=1,...,N,$$
where the functions $\psi_j$ are defined by
$$\psi_j = (\S_D^{0})^{-1}[\chi_{\p D_j}].$$

Observe that as $\delta \rightarrow 0$, we have $\lambda \rightarrow \tfrac{1}{2}^+$. Then, using Gohberg-Sigal theory for operator-valued functions \cite{AKL, gohberg2009holomorphic} we have the following lemma.

\begin{lem}
	For any $\delta$ sufficiently small there are, up to multiplicity, $N$ characteristic values
	$\omega_j= \omega_j(\delta), j = 1,...,N$, to the operator-valued analytic function 
	$\mathcal{A}(\omega, \delta)$
	such that $\omega_j(0)=0$ for all $j$ and 
	$\omega_j$ depends on $\delta$ continuously.
\end{lem}

The following theorem, proved in \cite{doublenegative}, shows that the eigenvalues of $C$ determine the resonance frequencies of the finite structure.
\begin{thm} \label{thm:char_approx_finite}
	The characteristic values $\omega_j=\omega_j(\delta),~j=1,...,N$, of $\mathcal{A}(\omega,\delta)$ can be approximated as
	$$ \omega_j= \sqrt{\frac{\delta \lambda_j }{|D_0|}}  + O(\delta),$$
	where $\lambda_j,~j=1,...,N$, are eigenvalues of the capacitance matrix $C$ and $|D_0|$ is the volume of a single resonator.
\end{thm}


\subsubsection{Nearest-neighbour approximation} \label{sec:finite_nearest_n}

Drawing on parallels to how SSH chains are studied in quantum mechanics, an appealing approach to approximating the problem of wave scattering by a finite system of subwavelength resonators is to consider a nearest-neighbour approximation. That is, to disregard long-range interactions between resonators, instead only considering the interactions between neighbouring elements. Mathematically, this means approximating the capacitance matrix \eqref{eq:finite_capacitance} by setting $C_{ij}=0$ if $|i-j|>1$, giving a tridiagonal matrix. Intuitively, one would expect that such an approach will only give a good estimate to the problem in the dilute regime.


%
%
We wish to prove estimates on the extent to which a tight-binding approach can approximate the problem in the dilute regime. We consider a dilute system by rescaling the canonical domain $D_0$ in \eqref{finite_form} as $D_0=\epsilon B$, where $B$ is some connected domain that has size of order one.
%
%
In this dilute regime, we are able to obtain an explicit representation of the capacitance matrix $C$ for the finite system \eqref{finite_form}. As in \Cref{subsec:cap_quasi}, we denote the capacitance of the fixed domain $B$ by $\textrm{Cap}_{B}$.

\begin{lem}\label{lem:cap_estim}
	Consider a dilute system of $N$ identical subwavelength resonators with size of order $\epsilon$, given by
	\begin{equation*}
		D=\bigcup_{j=1}^N \left(\epsilon B + z_j\right),
	\end{equation*}
	where $0<\epsilon\ll1$ and $z_j$ represents the fixed position of each resonator. In the limit as $\epsilon\rightarrow 0$ the asymptotic behaviour of the corresponding capacitance matrix is given by
	$$
	C_{ij} = 
	\begin{cases}
	\displaystyle \epsilon \mathrm{Cap}_B + O(\epsilon^3), &\quad i=j,\\
	\displaystyle -\frac{\epsilon^2(\mathrm{Cap}_B)^2}{4\pi|z_i - z_j|} + O(\epsilon^3), &\quad i\neq j.\\
	\end{cases}
	$$
\end{lem}

\begin{proof}
	The argument is very similar to that in \Cref{lem:cap_estim_quasi}.	We first write the single layer potential $\mathcal{S}_D^0$ in a decomposed matrix form, as
	\begin{align*}
	\mathcal{S}_D^0 &= \begin{pmatrix}
	\mathcal{S}_{D_1}^0 & 0 & \cdots & 0\\
	0 & \mathcal{S}_{D_2}^0 & \cdots & 0\\
	\vdots &  & \ddots & \vdots
	\\
	0 & \cdots & 0 & \mathcal{S}_{D_N}^0
	\end{pmatrix}
	+ 
	\begin{pmatrix}
	0 & \mathcal{S}_{D_2}^0|_{\p D_1} & \cdots & \mathcal{S}_{D_N}^0|_{\p D_1}\\
	\mathcal{S}_{D_1}^0 |_{\p D_2} & 0 & \cdots & \vdots \\
	\vdots &   & \ddots  & \mathcal{S}_{D_{N}}^0|_{\p D_{N-1}}
	\\
	\mathcal{S}_{D_1}^0|_{\p D_N} &  \cdots & \mathcal{S}_{D_{N-1}}^0|_{\p D_N} & 0
	\end{pmatrix}
	\\
	&:= S_I + S_{II}.
	\end{align*}
	We can then proceed to use scaling arguments, as in \Cref{lem:cap_estim_quasi}, to find estimates for $	\| S_I\|$, $\| S_I^{-1}\|$ and $ \| S_{II} \|$, and then use these bounds to compute the asymptotic behaviour of $C$.
\end{proof}

\begin{rmk}
	The explicit representations for $C_{ij}$ derived in \Cref{lem:cap_estim}, when used in the formula from \Cref{thm:char_approx_finite}, give approximations for the resonant frequencies in the dilute regime.  Moreover, the associated eigenmodes can be approximated using the fact that the characteristic functions $\Psi_j$, defined for each $\omega_j$ in \eqref{eq:boundary}, satisfy
	$$
	\Psi_j = \sum_{k=1}^N a_{jk} \psi_k + O(\sqrt{\delta}),
	$$
	where $(a_{j1}, \dots , a_{jN})$ is the eigenvector of $C$ associated with the eigenvalue $\lambda_j$. 
	This approach is particularly useful for performing efficient numerical computations.
\end{rmk}

	One can see from Lemma \ref{lem:cap_estim} that, for $i\neq j$, $C_{ij}$ satisfies the  slow decay property
	\begin{equation} \label{rmk:tb_is_bad}
		C_{ij} \sim \frac{1}{|i-j|}.
	\end{equation}
	This indicates that, for a system of subwavelength resonators, the nearest-neighbour approximation 
	may not give an accurate representation. This is a significant difference between the classical wave propagation problems studied here and the analogous applications of topological insulator theory in quantum mechanics, where nearest-neighbour approximations are commonplace.
	
	


\subsubsection{Chiral symmetry and edge mode frequencies} \label{subsec:chiral}

A prominent topic in the discussion of the SSH model is the notion of \emph{chiral symmetry}. This is a geometric property which a system is said to possess if there is an unitary matrix $\Sigma$ with $\Sigma^2=I$ such that the capacitance matrix $\widetilde{C}$ satisfies $\Sigma \widetilde{C} \Sigma = -\widetilde{C}$. The significance of this property is that a chirally symmetric matrix will have a symmetric spectrum. This is easily seen from the fact that if $(\lambda,v)$ is an eigenpair for $\widetilde{C}$ then so is $(-\lambda,\Sigma v)$. Finite chains that have an odd number of resonators (such as the example studied here, \Cref{fig:finite}) will have an odd number of resonant frequencies hence there must be a middle frequency. Thus, if one can design a chain which has a band gap (which we have a suggestion of how to do from \Cref{sec:inf_chain}) and is chirally symmetric, there must be a midgap frequency.

The reason we use the notation $\widetilde{C}$ for the capacitance matrix in this discussion is that in quantum mechanical settings it is customary to define the zero-energy state to be such that the diagonal entries of the Hamiltonian (which plays the analogous role of the capacitance matrix) vanish. Thus, one constructs a \emph{translated capacitance matrix} $\widetilde{C}$ by subtracting the constant diagonal elements. For the crystal in Figure \ref{fig:finite}, we can use \eqref{rmk:tb_is_bad} to approximate $\widetilde{C}$ by a nearest-neighbour approximation: a bisymmetric, tridiagonal matrix with odd size and zero diagonal. Such a matrix is chirally symmetric, and therefore has a zero eigenvalue. This shows that the finite system has a midgap frequency, at leading order.

The key property of a topologically protected state is that it retains its properties when imperfections exist in the structure. In particular, a chirally symmetric structure will retain its chiral symmetry when errors are made in the position of the resonators. This is because such errors will not affect the diagonal entries of $\widetilde C$ and, away from the diagonals, $\widetilde C_{ij}$ and $\widetilde C_{ji}$ will experience the same effects. Since the nearest-neighbour approximation of the capacitance matrix is chirally symmetric, we expect the midgap frequencies to be approximately stable with respect to errors in resonator position.

In \Cref{fig:tight_bind} we show how the resonant frequencies given by a nearest-neighbour approximation to a dilute resonator chain vary when subjected to errors in the position of the resonators. We use the multipole expansion method outlined in Appendix~\ref{app:multipole} to calculate the capacitance matrix \eqref{eq:finite_capacitance} then \Cref{thm:char_approx_finite} to compute the resonant frequencies from its eigenvalues. The pertinent conclusion from this is that, under the nearest-neighbour approximation, the midgap frequency is perfectly stable (as predicted by the above discussion). This approximation should be compared to \Cref{fig:dilute}, where the same simulations are performed on a fully-coupled chain. In light of the slow decay of the off-diagonal terms in the capacitance matrix (\ref{rmk:tb_is_bad}), the differences between the behaviour of the approximated and fully-coupled models are unsurprising, even when simulations are performed in a very dilute regime.

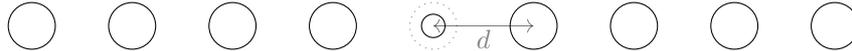
\begin{figure}[tb]
	\centering
	\begin{tikzpicture}[scale=1.1]
	\draw[dotted,opacity=0.5] (0,0) coordinate (start1) circle (8pt);
	\draw (0,0) coordinate (start1) circle (4pt);
	\draw[<->,opacity=0.5] (0,0) -- (1.2,0) node[pos=0.5,yshift=-5pt]{$d$};
	
	\begin{scope}[xshift=-1.2cm]
	\draw (0,0) coordinate (start1) circle (8pt);
	\end{scope}
	
	\begin{scope}[xshift=-2.4cm]
	\draw (0,0) coordinate (start1) circle (8pt);
	\end{scope}
	
	\begin{scope}[xshift=-3.6cm]
	\draw (0,0) coordinate (start1) circle (8pt);
	\end{scope}
	
	\begin{scope}[xshift=1.2cm]
	\draw (0,0) coordinate (start1) circle (8pt);
	\end{scope}
	
	\begin{scope}[xshift=2.4cm]
	\draw (0,0) coordinate (start1) circle (8pt);
	\end{scope}
	
	\begin{scope}[xshift=3.6cm]
	\draw (0,0) coordinate (start1) circle (8pt);
	\end{scope}
	
	\begin{scope}[xshift=4.8cm]
	\draw (0,0) coordinate (start1) circle (8pt);
	\end{scope}
	
	\begin{scope}[xshift=-4.8cm]
	\draw (0,0) coordinate (start1) circle (8pt);
	\end{scope}
	
	\end{tikzpicture}
	\caption{Two-dimensional cross-section of a finite subwavelength resonator chain with a point defect, which is expected to support an unprotected localized mode.} \label{fig:pointdefect}
\end{figure}

\begin{figure}
	\begin{center}
		\begin{subfigure}[b]{0.45\linewidth}
			\includegraphics[height=5cm]{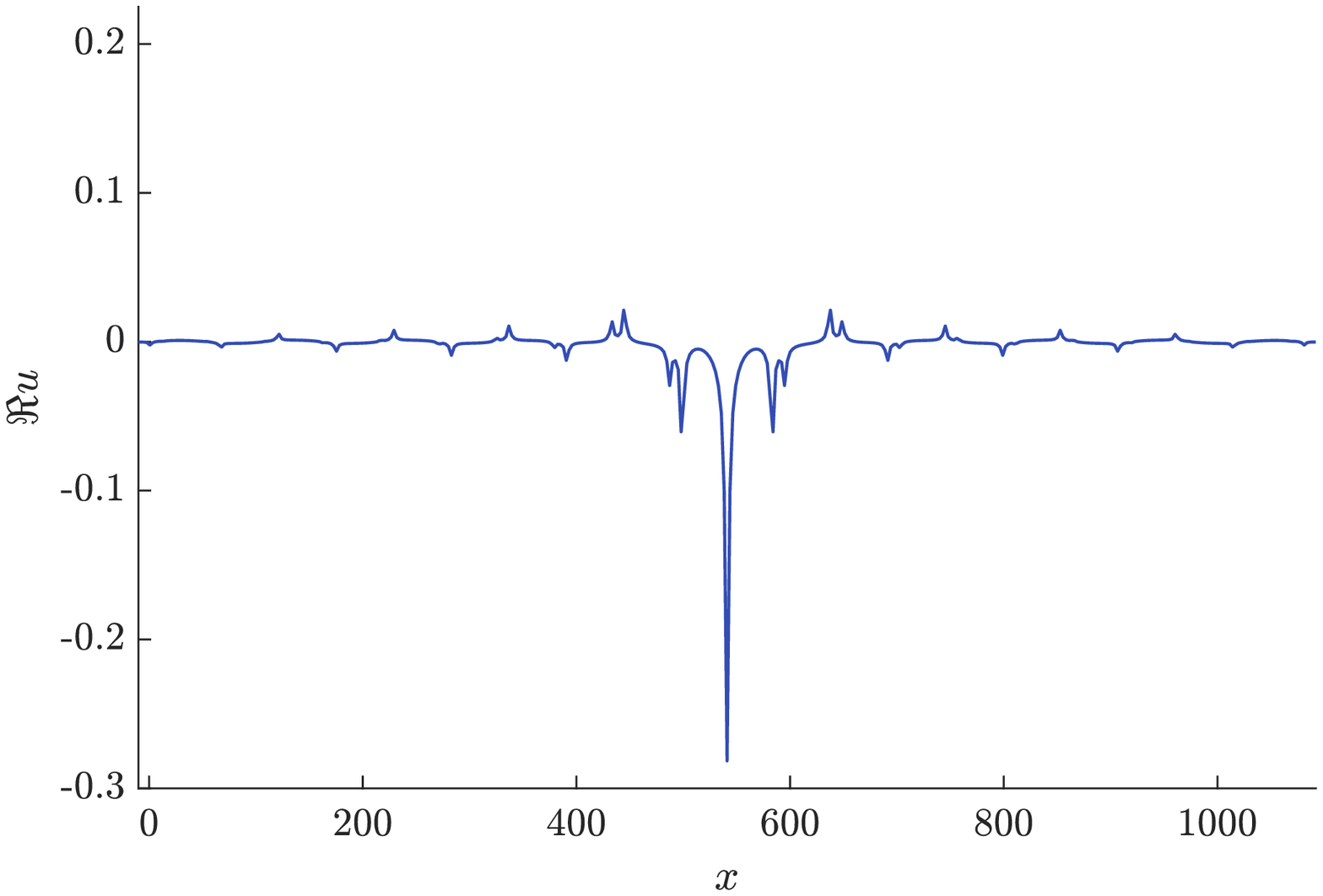}
			\caption{Topologically protected edge state for the (dilute) dimer chain. 
			} \label{fig:ssh_mode}
		\end{subfigure}
		\hspace{10pt}
		\begin{subfigure}[b]{0.45\linewidth}
			\includegraphics[height=5cm]{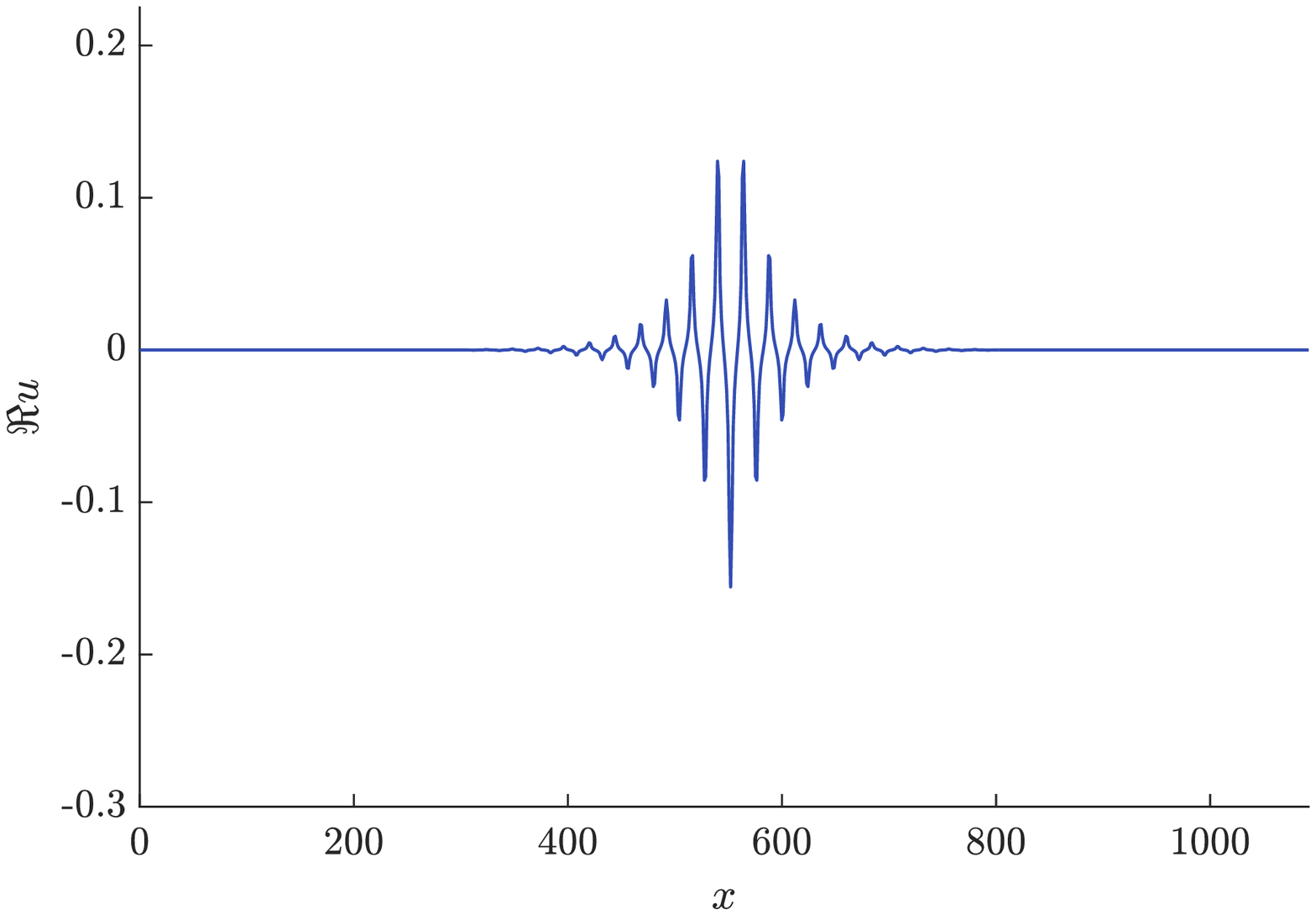}
			\caption{Unprotected localized mode for the point-defect chain.} \label{fig:defect}
		\end{subfigure}
		\caption{Comparison of the localised eigenstates exhibited by the finite chain of dimers (shown in \Cref{fig:finite}) and the point-defect chain (shown in \Cref{fig:pointdefect}). In both cases, a chain consisting of 41 resonators is used.}
	\end{center}
\end{figure}

\begin{figure}
	\begin{center}
		\begin{subfigure}[b]{0.45\linewidth}
			\includegraphics[height=5.0cm]{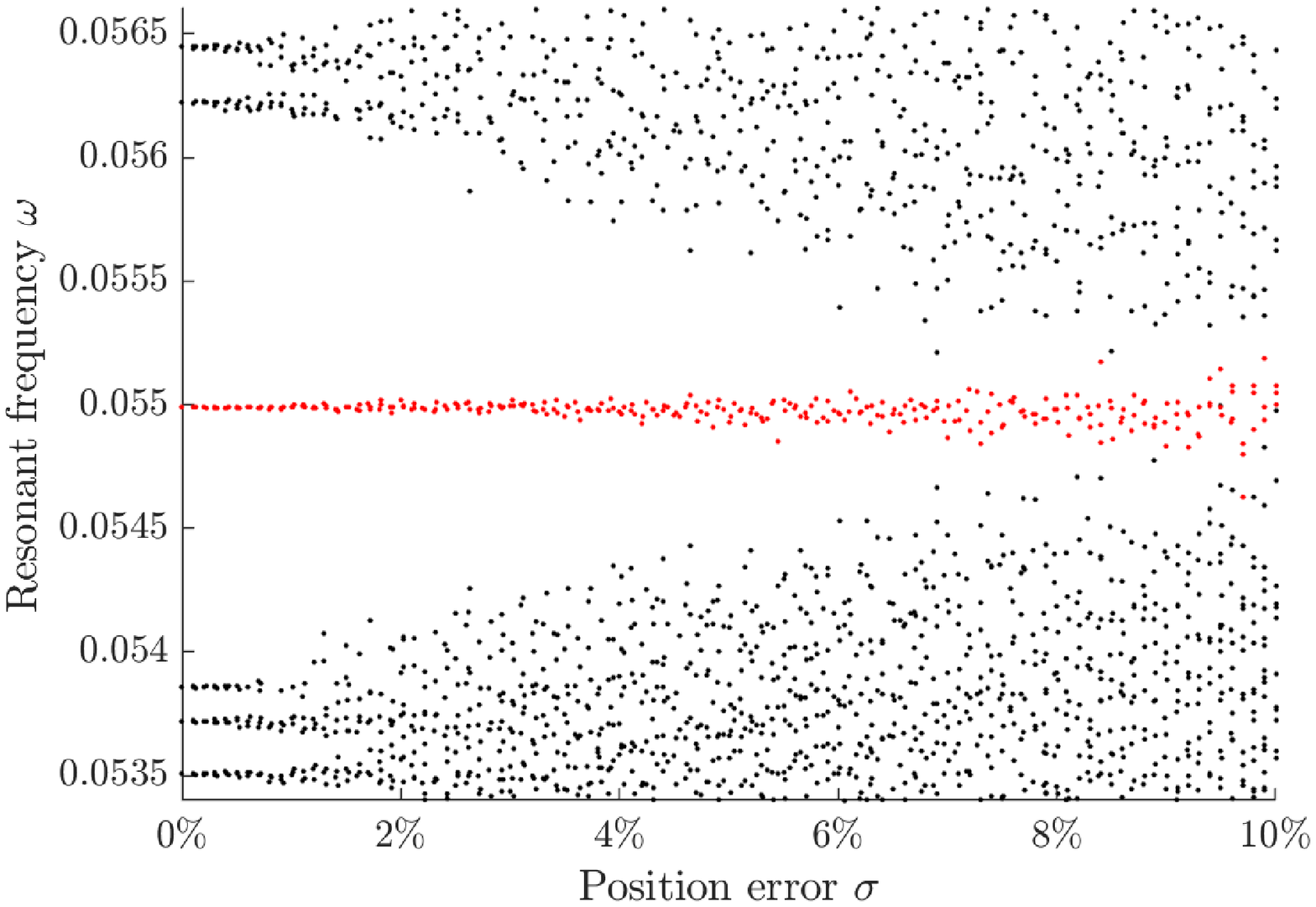}
			\caption{Dilute dimer chain with 41 resonators, separation distances $d=12, d' = 42$.} \label{fig:dilute}
		\end{subfigure}			
		\hspace{10pt}
		\begin{subfigure}[b]{0.45\linewidth}
			\includegraphics[height=5.0cm]{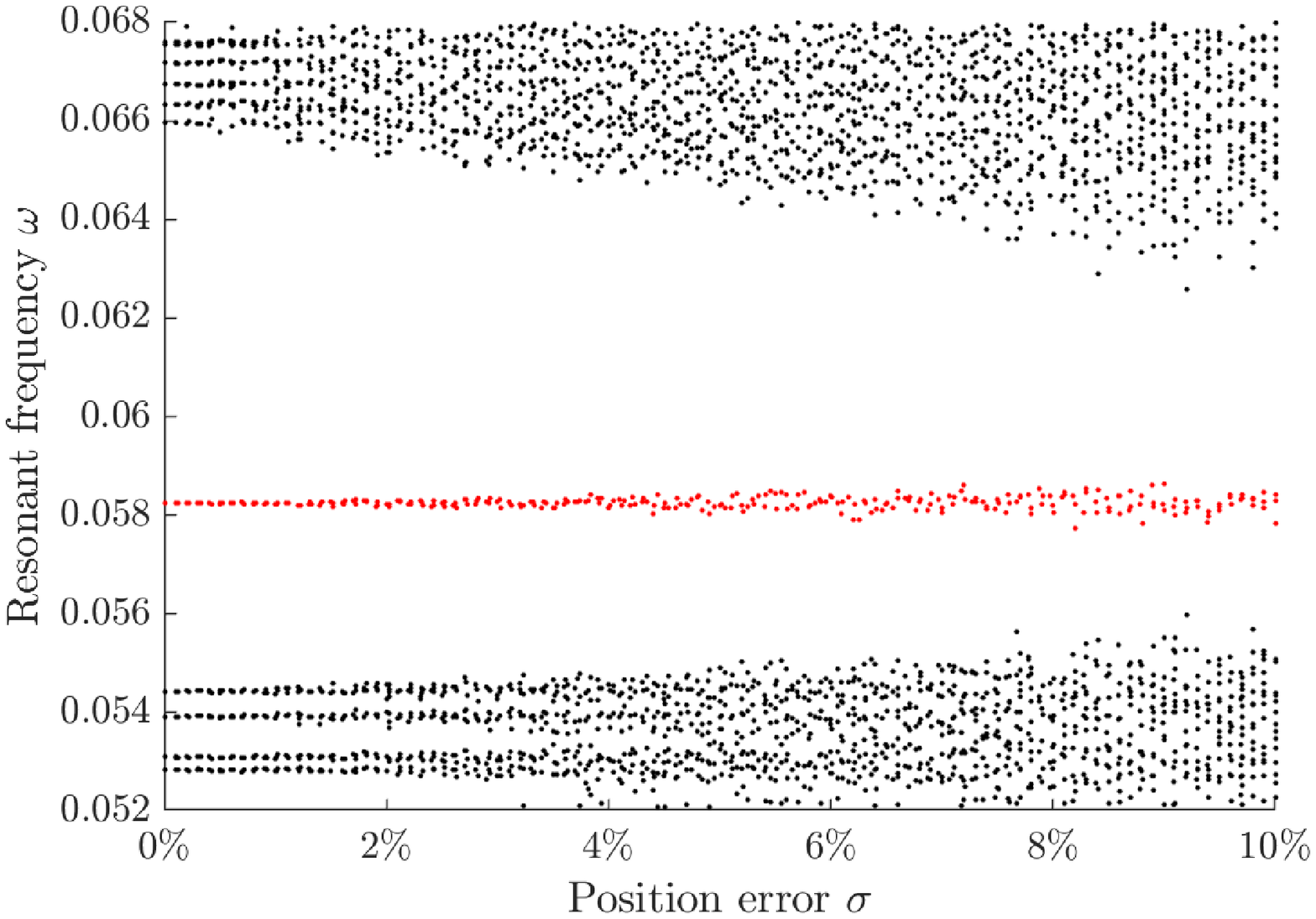}
			\caption{Non-dilute dimer chain with 41 resonators, separation distances $d=3, d' = 6$.} \label{fig:nondilute}
		\end{subfigure}	
		\vspace{20pt}
		\begin{subfigure}[b]{0.45\linewidth}
			\small
			\begin{center}
				\def\arraystretch{1.2}
				\begin{tabular}{ c | c | c }
					& dilute & non-dilute  \\ 
					\hline
					upper band & $1.03\times10^{-7}$ & $1.59\times10^{-7}$ \\  
					midgap & $4.90\times10^{-9}$ & $2.12\times10^{-8}$ \\
					lower band & $1.27\times10^{-7}$ & $7.67 \times10^{-7}$ 
					\vspace{15pt}
				\end{tabular}
			\end{center}
			\caption{The variance of the midgap and bulk frequencies from plots (a) and (b), for $\sigma=8\%$.} \label{var_table}
		\end{subfigure}
		\hspace{10pt}
	\begin{subfigure}[b]{0.45\linewidth}
		\includegraphics[height=4.1cm]{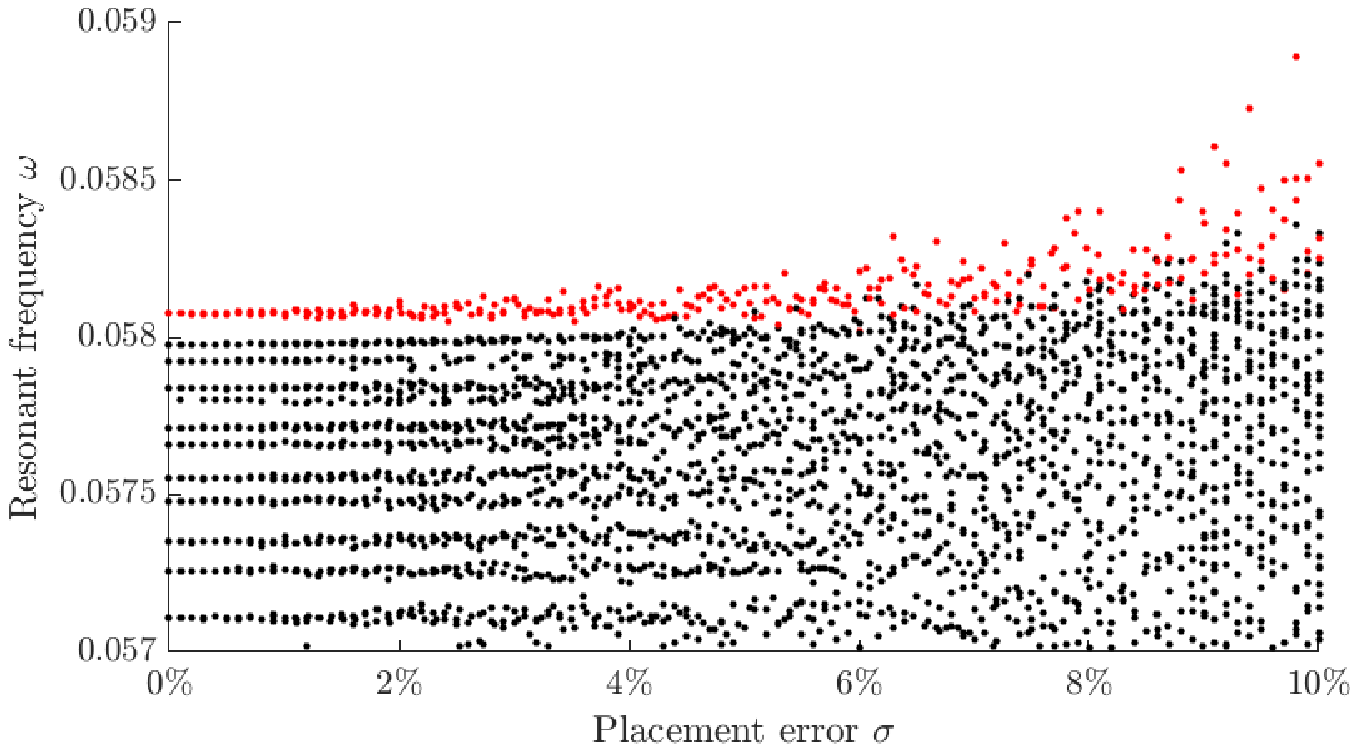}
		\caption{Dilute point-defect chain with 41 resonators, separation distance $d=12$ and defect radius $R_d = 0.99$.} \label{fig:point_defect_errors}
	\end{subfigure}
		\vspace{20pt}
		\begin{subfigure}[b]{0.45\linewidth}
			\includegraphics[height=5.0cm]{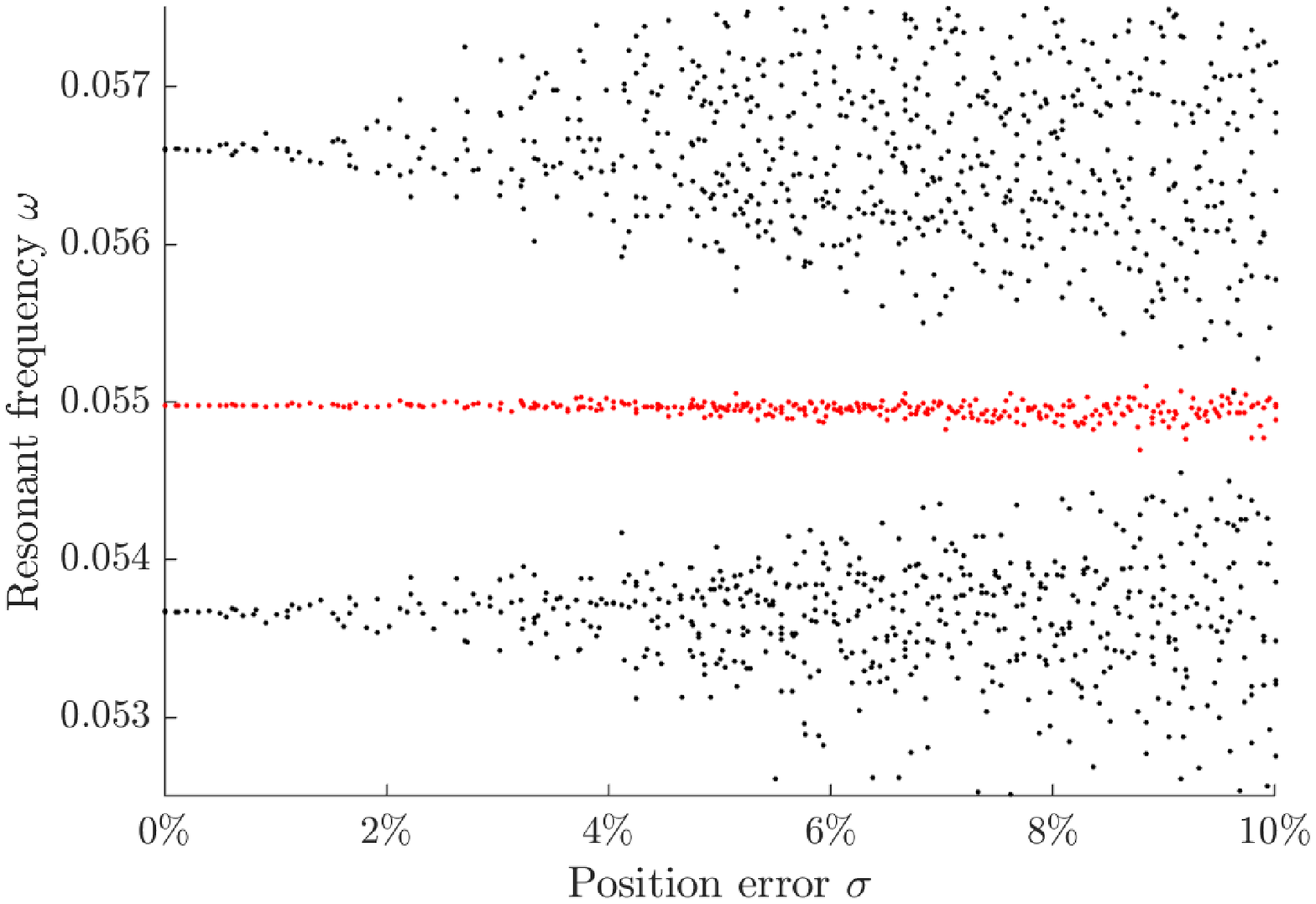}
			\caption{Dilute dimer chain with 9 resonators, separation distances $d=12, d' = 42$.} \label{fig:nine}
		\end{subfigure}	
		\hspace{10pt}
		\begin{subfigure}[b]{0.45\linewidth}
			\includegraphics[height=5.0cm]{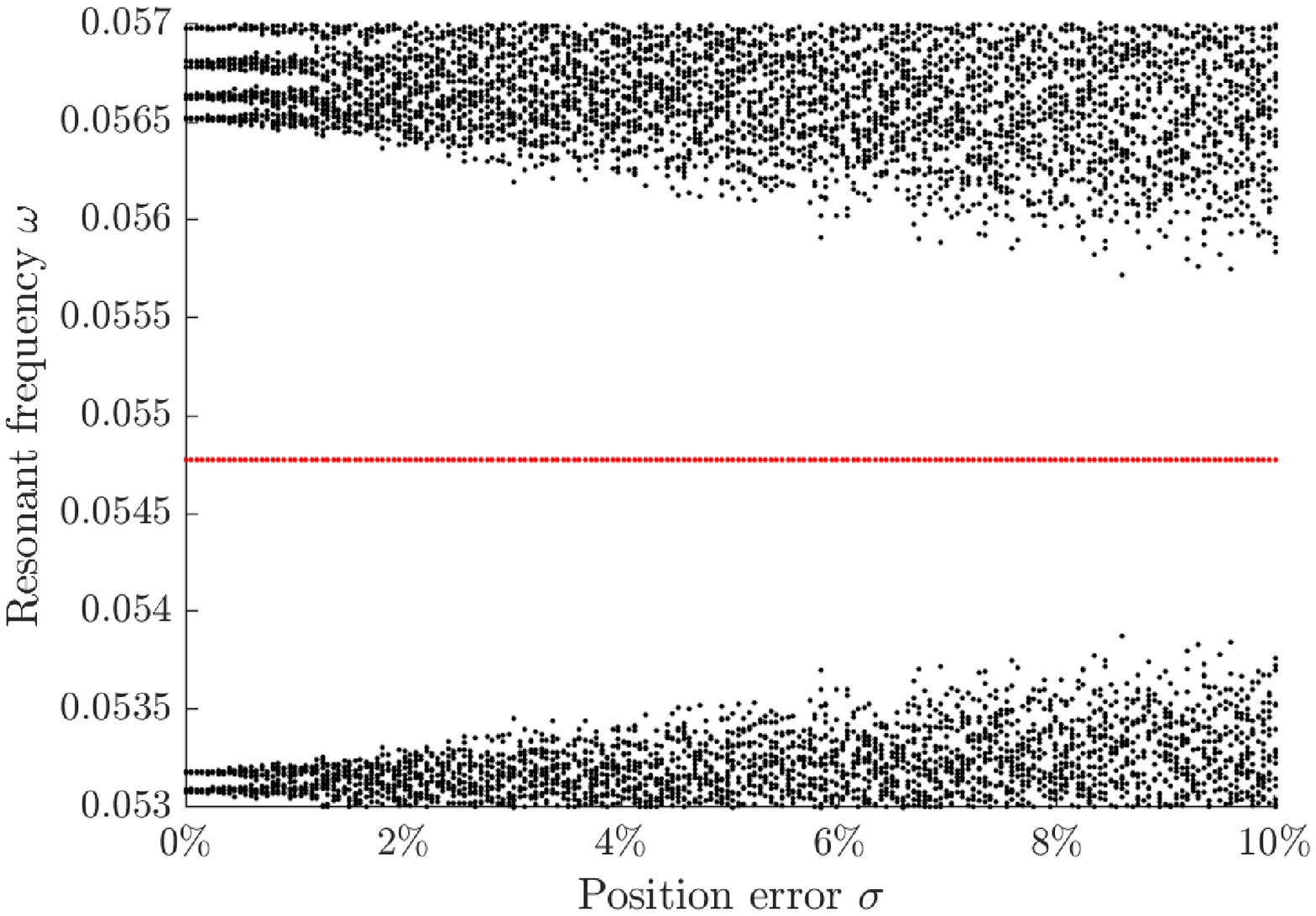}
			\caption{Nearest-neighbour approximation for the dilute dimer chain from (a).
			} \label{fig:tight_bind}
		\end{subfigure}	
		\caption{Simulation of band gap frequency (red) and bulk frequencies (black) of different subwavelength resonator chains with Gaussian $\mathcal{N}(0,\sigma^2)$ errors added to the resonator positions. The standard deviation $\sigma$ is expressed as a percentage of the average resonator separation. 
		}
		\label{fig:perturbation}		
	\end{center}
\end{figure}

\subsection{Numerical illustrations} \label{sec:finite_numerics}
We now perform a series of numerical computations to illustrate the difference between the topologically protected subwavelength localized modes in the finite dimer chain \eqref{finite_form} and conventional, unprotected, subwavelength localized modes. The unprotected mode we study is produced by taking an equally spaced chain of resonators and changing the radius of the central resonator, thus introducing a defect (often known as a \emph{point defect}). This system, depicted in \Cref{fig:pointdefect}, is the finite, one-dimensional equivalent of the system studied in \cite{defectSIAM}, where the existence of a subwavelength localized mode was proved in the case of an infinite crystal.

As was the case for the infinite chain in \Cref{sec:inf_numerics}, the following numerical results for the finite chains were calculated for the case of acoustic waves being scattered by (subwavelength) air bubbles in water. The details of discretizing the operator $\A(\omega,\delta)$ using the multipole expansion method are given in Appendix~\ref{app:multipole}.

\subsubsection{Existence of localized modes} \label{sec:finite_modes}
\Cref{fig:ssh_mode,,fig:defect} show the localized modes for the dimer and point-defect chains respectively (whose geometries are depicted in \Cref{fig:finite,,fig:pointdefect}). The configurations have been chosen to give roughly the same strength of the localization.

\subsubsection{Stability with respect to errors}
Finally, we study the stability of the edge mode frequency with respect to random, symmetry-preserving imperfections. In light of the discussion in \Cref{subsec:chiral}, we add random errors to the positions of the resonators and repeatedly compute the resonant frequencies. In \Cref{fig:dilute,,fig:nondilute} we can see that, in both the dilute and non-dilute regimes, the structure supports a localized mode (depicted in \Cref{fig:ssh_mode} for the dilute case) whose resonant frequency is in the middle of the band gap. In Table~\ref{var_table} it is demonstrated that in the two regimes the stability of each frequency with respect to the random errors is very similar in magnitude. The fact that the midgap frequency is consistently further from the edges of the band gap in the non-dilute case is merely a consequence of the gap being wider in this regime.
In \Cref{fig:dilute} we present the same simulations for a very short dimer chain, with only nine resonators. We can see, once again, that there is a midgap frequency which is much more stable than the bulk frequencies. 

Finally, we make a comparison with the conventional defect mode exhibited by the subwavelength point-defect chain (shown in \Cref{fig:pointdefect}). It is clear from \Cref{fig:point_defect_errors} that, even for relatively small errors, the frequency associated with the point-defect mode exhibits poor stability and is easily lost amongst the bulk frequencies. The comparison between the robustness of the two designs is particularly eye-opening in light of the observation that the degree of wave localization is very similar. The new, dimerized design is equally capable of localizing waves at subwavelength scales but does so with spectacularly enhanced robustness.

\section{Concluding remarks}
In this work, we have, both analytically and numerically, studied a fully-coupled chain of subwavelength resonator dimers. We have shown that the infinite crystal exhibits a non-trivial Zak phase in certain resonator configurations. In the dilute regime, we have given explicit expressions for the Zak phase, proved the existence of a non-trivial band gap and shown that band inversion occurs between the two different phase regimes. Guided by these findings, we have designed a finite resonator chain that exhibits topologically protected edge modes at its centre. This was based on being able to associate the dimers on either side of this edge with different values of the Zak phase. We have shown numerically that the edge mode frequency is well-localized in the band gap and that, when errors are added to the positions of the resonators, the variance of this frequency is significantly lower than that of the bulk frequencies. 
Although much of the explicit analysis was performed on infinite chains, numerical experiments showed that our approach can be used to create topologically protected edge modes in structures that contain only very small numbers of resonators.




\bibliographystyle{abbrv}
\bibliography{tightbind_jmpa}{}
	
\begin{appendices}
	\section{Multipole expansion method in three dimensions} \label{app:multipole}
	Here we derive the multipole expansion approximation of $\S_D^k$ and $\S_D^{\alpha,k}$ in three dimensions. The method is a generalization of the method in two dimensions given in Appendix C of \cite{bandgap}. The overarching principle is that when working on spherical domains, the action of the single layer potential on spherical basis functions has an explicit, analytic representation.
	
	The goal is to discretize the equations \eqnref{eq:boundary_quasi} and \eqnref{eq:boundary}. Observe that the operators $\A$ and $\A^\alpha$ can be written as
	$$
	\A(k,\delta) = \frac{\partial }{\partial \nu} \S_D^k \Big|_{-}- \delta \frac{\partial }{\partial \nu} \S_D^k\Big|_{+},
	$$
	and
	$$
	\A^\alpha(k,\delta) = \frac{\p}{\p\nu}  \mathcal{S}_D^{\alpha,k}\Big|_{-} - \delta \frac{\p}{\p\nu} \mathcal{S}_D^{\alpha,k} \Big|_{+},
	$$
	so it is enough to find a discretized representation of the single layer potentials $\S_D^k$ and $\S_D^{\alpha,k}$. 
	
	For a radially symmetric Helmholtz equation, it is well-known that the spherical waves $j_l(kr)Y_l^m(\theta, \phi)$ and $h_l^{(1)}(kr)Y_l^m(\theta,\phi)$ gives a basis of solutions in the polar coordinates $(r,\theta,\phi)$. Here $Y_l^m(\theta, \phi), l\in \N, m = -l,...,l$, are the spherical harmonics and $j_n(kr), h_n^{(1)}(kr)$ are the spherical Bessel and Hankel functions of the first kind, respectively, defined by 
	$$j_l(x) = \sqrt{\frac{\pi}{2x}}J_{l+\frac{1}{2}}(x), \quad h_l^{(1)}(x) = \sqrt{\frac{\pi}{2x}}H_{l+\frac{1}{2}}^{(1)}(x),$$
	where $J_n$ and $H_n^{(1)}$ are the ordinary Bessel and Hankel functions of the first kind.
	
	We begin by deriving the multipole expansion of the single layer potential $\S_D^k$. The spherical harmonics $Y_l^m$ form a basis of $L^2(\p D)$, and we seek the expansion of $\S_D^k$ in this basis. Define $u:=\mathcal{S}_D^{k}[Y_l^m]$, which is the solution to 
	\begin{equation} 
	\left\{
	\begin{array}{ll}
	\ds  \Delta u + k^2 u=0 \quad  &\mbox{in } \mathbb{R}^3\setminus
	\overline{D}, \\
	\nm
	\ds \Delta u + k^2  u=0 \quad  &\mbox{in } D, \\
	\nm
	\ds u|_{+}=u|_{-} \quad &\mbox{on } \p D, \\
	\nm
	\ds  \frac{\p u}{\p\nu} \Big|_{+} -  \frac{\p u}{\p\nu} \Big|_{-} = Y_l^m \quad &\mbox{on } \p D, \\
	\ds |x| \left(\tfrac{\p}{\p|x|}-\iu k\right)u \to 0 
	&\text{as } {|x|} \rightarrow \infty.
	\end{array}
	\right.
	\end{equation}
	The above equation can be easily solved by the separation of variables technique in polar coordinates. It gives
	\begin{equation}
	\label{SingleLayer_multipole}
	\mathcal{S}_D^{k}[Y_l^m](r,\theta,\phi) = 
	\begin{cases}
	\ds cj_l(kR) h_l^{(1)}(kr)Y_l^m(\theta,\phi), &\quad |r|>R,
	\\[0.5em]
	\ds ch_l^{(1)}(kR)j_l(kr)Y_l^m(\theta, \phi), &\quad |r|\leq R,
	\end{cases}
	\end{equation}
	where $c=-\iu  k R^2$.
	
	In order to handle problems posed on disjoint domains, we will need an addition theorem relating spherical waves centred around a translated origin to spherical waves around the original origin.	Suppose we have a point with coordinates $x = (r,\theta,\phi)$ in the original system and $x' = (r',\theta',\phi')$ in the translated system, with the coordinate vectors related by $x = x' + b$ for $b = (r_b,\theta_b,\phi_b)$. Moreover, we assume $r'<r_b$. Then the addition theorem reads \cite{addition}
	\begin{equation}\label{addition_theorem}
	 h_l^{(1)}(kr)Y_{l}^{m}(\theta,\phi) = \sum_{l' \in \N, |m'|\leq l'} A_{l'm'}^{lm}j_{l'}(kr')Y_{l'}^{m'}(\theta',\phi'),
	\end{equation}	
	where the coefficients $A_{l'm'}^{lm}$ are given by 
	$$A_{l'm'}^{lm} = \sum_{\lambda\in \N,|\mu|\leq \lambda} C(l,m,l',m',\lambda,\mu) h_\lambda^{(1)}(kr_b)Y_\lambda^\mu(\theta_b, \phi_b).$$
	Here, the coefficients $C(l,m,l',m',\lambda,\mu)$ are in turn given by
	$$C(l,m,l',m',\lambda,\mu) = i^{l'-l+\lambda}(-1)^{m}\sqrt{4\pi(2l+1)(2l'+1)(2\lambda+1)}
	\begin{pmatrix} l & l' & \lambda \\ 0 & 0 & 0\end{pmatrix}
	\begin{pmatrix} l & l' & \lambda \\ -m & m' & \mu\end{pmatrix},$$
	where we denote by 
	$$\begin{pmatrix} j_1 & j_2 & j_3 \\ m_1 & m_2 & m_3\end{pmatrix},$$
	the Wigner $3j$ symbols. To simplify these expressions slightly, we assume that the original coordinate system is aligned such that $b$ points along the positive $z$-axis, \ie{}, $\theta_b = 0$. In this case 
	$$ Y_\lambda^\mu(\theta_b, \phi_b) = \begin{cases} 0, \qquad &\mu \neq 0, \\ \sqrt{\frac{2\lambda + 1}{4\pi}}, \quad &\mu = 0. \end{cases}$$
	Substituting this into the expression for $A_{l'm'}^{lm}$ gives
	$$A_{l'm'}^{lm} = \sum_{\lambda \in \N}\sqrt{\frac{2\lambda + 1}{4\pi}} C(l,m,l',m',\lambda,0) j_\lambda(kr_b).$$
	
	Now, we compute the quasiperiodic single layer potential $\mathcal{S}_{D}^{\alpha,k}[Y_l^m]$ in the case when $D$ consists of a single resonator. Since 
	$$ 
	G^{\alpha,k} (x,y) = \sum_{n \in \mathbb{Z}} G^k(|x-y-(nL,0,0)|)e^{\iu n\alpha L},$$
	we have
	\begin{align*}
	\mathcal{S}_D^{\alpha,k}[Y_l^m](x)
	&=\mathcal{S}_D^{k}[Y_l^m](x)+\sum_{n\in\mathbb{Z}, n\neq 0} \mathcal{S}^{k}_{D+n}[Y_l^m]e^{\iu  n\alpha L}
	\\
	&=\mathcal{S}_D^{k}[Y_l^m](x)+cj_n(kR)\sum_{n\in\mathbb{Z}, n\neq 0} h_l^{(1)}(kr_n')Y_l^m(\theta_n',\phi_n') e^{\iu  n\alpha}.
	\end{align*}
	Here, $D+n$ means a translation of the disk $D$ by $(nL,0,0)$ and $(r_n',\theta_n',\phi_n')$ are the spherical coordinates with respect to the centre of $D+n$.
	
	Using the addition theorem \eqref{addition_theorem} we have
	\begin{align*}
	\mathcal{S}_D^{\alpha,k}[Y_l^m](x)
	=& \mathcal{S}_D^{k}[Y_l^m](x)+cj_l(kR)\sum_{l' \in \N, |m'|\leq l'}\left[ \sum_{\lambda\in \N,|\mu|\leq \lambda}C(l,m,l',m',\lambda,\mu) Q_\lambda^\mu\right] j_{l'}(kr)  Y_{l'}^{m'}(\theta,\phi)\\
	:=&  \mathcal{S}_D^{k}[Y_l^m](x)+cj_l(kR)\sum_{l' \in \N, |m'|\leq l'}B_{l'm'}^{lm} j_{l'}(kr)  Y_{l'}^{m'}(\theta,\phi),
	\end{align*}
	where $Q_\lambda^\mu$ is the one-dimensional lattice sum in three dimensions, defined by
	$$Q_\lambda^\mu = \sum_{n\in\mathbb{Z}, n\neq 0} h_\lambda^{(1)}(knL)Y_\lambda^\mu(\theta_n, \phi_n)e^{\iu  n\alpha L}.$$
	An efficient method for computing this lattice sum is given in \cite{linton}.
	
	We are now ready to compute $\S_D^{\alpha,k}$ in the case when $D\Subset Y$ consists of two resonators, centred at $(-x_1,0,0)$ and $(x_1,0,0)$, respectively. This is what we require in order to perform computations on the infinite chain in \Cref{sec:inf_numerics}. By identifying $L^2(\p D) = L^2(\p D_1) \times L^2(\p D_2)$ we have 
	$$\S_D^{\alpha,k} = \begin{pmatrix} \S_{D_1}^{\alpha,k} & \S_{D_2}^{\alpha,k}\big|_{\p D_1} \\[0.5em] 
	\S_{D_1}^{\alpha,k}\big|_{\p D_2} & \S_{D_2}^{\alpha,k}\end{pmatrix}.$$
	Here the operator $\S_{D_i}^{\alpha,k}\big|_{\p D_j}: L^2(\p D_i) \rightarrow L^2(\p D_j), i,j = 1,2$ is the evaluation of $\S_{D_i}^{\alpha,k}$ on $\p D_j$. To compute the multipole expansion of $\S_{D_1}^{\alpha,k}\big|_{\p D_2}$, we again use the addition theorem. We have
	\begin{align*}
	&\S_{D_1}^{\alpha,k}\big|_{\p D_2}[Y_l^m](x') =  cj_l(kR) h_l^{(1)}(kr')Y_l^m(\theta',\phi')+cj_l(kR)\sum_{l' \in \N, |m'|\leq l'}B_{l'm'}^{lm} j_{l'}(kr')  Y_{l'}^{m'}(\theta',\phi') \\
	&\quad=cj_l(kR) \sum_{l'' \in \N, |m''|\leq l''} \left[ \sum_{\lambda\in \N,|\mu|\leq \lambda} C(l,m,l'',m'',\lambda,\mu) h_\lambda^{(1)}(kd)Y_\lambda^\mu(\theta_d, \phi_d) \right]j_{l''}(kr)Y_{l''}^{m''}(\theta,\phi) \\
	&\quad\qquad  + cj_l(kR)\sum_{l'' \in \N, |m''|\leq l''}\left[\sum_{\substack{l' \in \N, |m'|\leq l' \\ \lambda\in \N,|\mu|\leq \lambda}}B_{l'm'}^{lm}  C(l',m',l'',m'',\lambda,\mu) j_\lambda(kd)Y_\lambda^\mu(\theta_d, \phi_d)\right]j_{l''}(kr)Y_{l''}^{m''}(\theta,\phi).
	\end{align*}
	
	In order to simulate the finite resonator chain in \Cref{sec:finite_numerics}, we must now perform similar computations for the operator $\S_{D}^k$ in the case when $D$ consists of $N$ resonators. We assume the resonators to be arranged collinearly along the $x_1$-axis. By identifying $L^2(\p D) = L^2(\p D_1) \times \ldots\times L^2(\p D_N)$ we have 
	$$\S_D^{k} = \begin{pmatrix} 
	\S_{D_1}^{k} & \S_{D_2}^k\big|_{\p D_1} & \dots & \S_{D_N}^k\big|_{\p D_1}  \\[0.5em]  
	\S_{D_1}^k\big|_{\p D_2} & \S_{D_2}^{k} & \dots & \S_{D_N}^k\big|_{\p D_2}\\
	\vdots & \vdots & \ddots & \vdots \\
	\S_{D_1}^k\big|_{\p D_N} & \S_{D_2}^k\big|_{\p D_N} & \dots & \S_{D_N}^{k}
	\end{pmatrix},$$
	where, as in the quasiperiodic case, $\S_{D_i}^k\big|_{\p D_j}: L^2(\p D_i) \rightarrow L^2(\p D_j)$ is the evaluation of $\S_{D_i}^{k}$ on $\p D_j$. This relies on the addition theorem once again. The diagonal terms are easily evaluated using \eqref{SingleLayer_multipole} directly. Away from the diagonals, the addition theorem \eqref{addition_theorem} gives that
	\begin{equation*}
	\S_{D_j}^k\big|_{\p D_i}[Y_l^m](x') = ch_l^{(1)}(kR)\sum_{l' \in \N, |m'|\leq l'} A_{l'm'}^{lm}j_{l'}(kr')Y_{l'}^{m'}(\theta',\phi').
	\end{equation*}
	
	\section{Proof of \Cref{lem:c=0}} \label{app:lemma}
	\noindent \textbf{Part (i): $\text{Im }C_{12}^\alpha > 0$ for $0<\alpha<\pi/L$ and $\text{Im }C_{12}^\alpha < 0$ for $-\pi/L<\alpha<0$. In particular, $\text{Im }{C_{12}^{\alpha}}$ is zero if and only if $\alpha \in\{ 0, \pi/L \}$.}
\smallskip

\noindent
Recall, from  \Cref{lem:cap_estim_quasi}, the following expansion of ${C_{12}^{\pi/L}}$ for fixed $\alpha \neq 0$ in the dilute regime:
\begin{align}C_{12}^\alpha =& -\frac{(\epsilon \mathrm{Cap}_B)^2}{4\pi}\sum_{m =-\infty}^\infty \frac{e^{\iu m \alpha L} }{  |d + mL| } + O(\epsilon^3) \nonumber \\
:=& -\frac{(\epsilon \mathrm{Cap}_B)^2}{4\pi L} f(\alpha,d) + O(\epsilon^3), \label{eq:c12d}
\end{align}
where, in order to simplify notation, we have defined the function $f$ as
\begin{equation} \label{eq:f}
f(\alpha,d) := \sum_{m =-\infty}^\infty \frac{e^{\iu m \alpha L} }{  |d/L + m| }.
\end{equation}
The imaginary part 
\begin{equation*}
\text{Im }f(\alpha,d) = \sum_{m =-\infty}^\infty \frac{\sin(m\alpha L) }{  |d/L + m| }
\end{equation*}
converges for all $\alpha \in Y^*$, and \eqnref{eq:c12d} is valid for imaginary parts also at $\alpha = 0$.

We will express $f$ in terms of Lerch's Transcendent function $\Phi(z,s,a)$, after having first reviewed some basic properties. For details we refer to \cite{bateman1953higher}. $\Phi(z,s,a)$ is defined by the power series
\begin{equation}\label{eq:phi}
\Phi(z,s,a) = \sum_{m=0}^\infty \frac{z^m}{(a+m)^s},
\end{equation}
for $z\in \mathbb{C}$ where this series converges, and extended by analytic continuation. If $\text{Re}(s) > 0, \text{Re}(a) > 0$ and $z\in \mathbb{C}\setminus [1,\infty)$ we have the integral representation
\begin{equation}\label{eq:phi_int}
\Phi(z,s,a) = \frac{1}{\Gamma(s)}\int_0^\infty \frac{t^{s-1}e^{-at}}{1-ze^{-t}}\dx t,
\end{equation}
where $\Gamma$ is the Gamma function.

Now, from the definition of $f$ in \eqnref{eq:f}, we have 
\begin{align*}f(\alpha,d) = \Phi(e^{\iu \alpha L},1,d/L) + e^{-\iu \alpha L}\Phi(e^{-\iu \alpha L},1,1-d/L),
\end{align*}
Using the integral representation \eqnref{eq:phi_int} we get, after simplifications,
\begin{equation} \label{eq:fint}
f(\alpha, d) = \int_0^\infty \frac{e^{-\iu \alpha L}\sinh\left(\frac{d}{L}t\right) + \sinh\left(\frac{1-d}{L}t\right)}{\cosh(t) - \cos(\alpha L)}\dx t.
\end{equation}
The imaginary part satisfies
\begin{equation} \label{eq:Imf}
\text{Im } f(\alpha, d) = \sin(\alpha L)\int_0^\infty \frac{\sinh\left(\frac{d}{L}t\right)}{\cos(\alpha L) - \cosh(t)}\dx t.
\end{equation}

At the points $\alpha = 0$ and $\alpha = \pi/L$, the functions $V_1^\alpha$ and $V_2^\alpha$ are real-valued and hence $\text{Im } C_{12}^\alpha = 0$. We will show that, for $\epsilon$ small enough, this imaginary part is zero precisely when $\alpha \in\{ 0, \pi/L \}$. The integrand in \eqnref{eq:Imf} is positive, and hence $\text{Im } f(\alpha, d) = 0$ if and only if $\alpha \in\{ 0, \pi/L \}$ for $\alpha \in Y^* = (-\pi/L, \pi/L]$. This shows that the leading order term of $\text{Im } C_{12}^\alpha$ is zero precisely when $\alpha \in\{ 0, \pi/L \}$. Moreover, we can easily verify from \eqnref{eq:Imf} that
$$
\frac{\p}{\p\alpha}\text{Im } f(\alpha,d)\big|_{\alpha = 0} \neq 0, \qquad \frac{\p}{\p\alpha}\text{Im } f(\alpha,d)\big|_{\alpha = \pi/L^-} \neq 0, \qquad \frac{\p}{\p\alpha}\text{Im } f(\alpha,d)\big|_{\alpha = -\pi/L^+} \neq 0.
$$
This shows that for small enough $\epsilon$, the function $\text{Im }C_{12}^\alpha$ will be monotonic around $\alpha = 0$ and $\alpha = \pi/L$. Since we know that $\alpha \in\{ 0, \pi/L \}$ are exact zeros of $\text{Im }C_{12}^\alpha,$ these zeros will be isolated for small enough $\epsilon$. It follows that,  for small enough $\epsilon$, $\text{Im }C_{12}^\alpha$ is zero precisely when $\alpha \in\{ 0, \pi/L \}$. Then, from \eqnref{eq:Imf} it follows that $\text{Im }C_{12}^\alpha > 0$ for $0<\alpha<\pi/L$ and $\text{Im }C_{12}^\alpha < 0$ for $-\pi/L<\alpha<0$.



\smallskip

\noindent \textbf{Part (ii): $C_{12}^{\alpha}$ is zero if and only if both $d = d'$ and $\alpha = \pi/L$.}
\smallskip

\noindent
By part (i), any zero must satisfy $\alpha = 0$ or $\alpha = \pi/L$. We begin by excluding the case $\alpha = 0$. As $\alpha \rightarrow 0$, it is known that the quasiperiodic capacitance of a single particle vanishes \cite{MaCMiPaP,bandgap}. In other words, we have, for the total capacitance of $D$,
\begin{align*}
0 & = \int_{Y\setminus D} \nabla(V_1^0 + V_2^0)\cdot\overline{\nabla(V_1^0+V_2^0)} \dx x \\
&= 2(C_{11}^0+C_{12}^0),
\end{align*}
where the last equality follows since $C_{12}^0$ is real. Since $C_{11}^\alpha >0$ we have $C_{12}^0 < 0$.

We now turn to the case $\alpha = \pi/L$. We already know from \Cref{lem:cc'} that ${C_{12}^{\pi/L}} = 0$ if $d=d'$. To show that this is the only zero of ${C_{12}^{\pi/L}}$, we will show that ${C_{12}^{\pi/L}}$ is strictly monotonic as a function of $d<L$. From the definition of $f$ we compute the derivative
$$\frac{\p}{\p d}f(\alpha,d)\bigg|_{\alpha=\pi/L} = -\frac{L}{d^2} + \sum_{m =1}^\infty \frac{(-1)^m}{L}\left[\frac{1}{(m-d/L)^2}-\frac{1}{(m+d/L)^2} \right].$$
The sum is an alternating series, with decreasing terms and negative first term. Hence it converges to a negative value, and by \eqnref{eq:c12d} we have   
\begin{equation}
\label{eq:deri}
\frac{\p}{\p d} {C_{12}^{\pi/L}} > 0,
\end{equation}
for $\epsilon$ small enough. This shows that ${C_{12}^{\pi/L}}$ has a unique zero when $d=d'$, which completes the proof of part (ii).
\smallskip

\noindent \textbf{Part (iii): $C_{12}^{\pi/L} < 0$  when $d<d'$ and $C_{12}^{\pi/L} > 0$ when $d>d'$. In both cases we have $C_{12}^{0} < 0$.}
\smallskip

\noindent We already know, from the proof of part (ii), that ${C_{12}^{0}} < 0$. The other conclusions, namely that ${C_{12}^{\pi/L}} < 0$ if $d<d'$ and ${C_{12}^{\pi/L}} > 0$ if $d>d'$, follow directly from \eqnref{eq:deri}.
\qed

\section{Proof of \Cref{thm:band_gap}} \label{app:lemma_gap}
We begin by proving the following lemma.
\begin{lem} \label{lem:ineq}
	For every $a,b$ such that $-1\leq a<1$ and $0<b<1$, the following holds:
	\begin{equation} \label{eq:ineq}
	\int_0^\infty \frac{\sinh(bt)}{\cosh(t)-a}\dx t > -\frac{2b}{1+a}\log\left(\frac{1}{2}(1-a)\right).
	\end{equation}
\end{lem} 
\begin{proof}
	We will split into the cases $a>0$ and $a\leq 0$, and begin with $a>0$. The right-hand side can be written as follows:
	\begin{equation} \label{eq:log}
	-\frac{2b}{1+a}\log\left(\frac{1}{2}(1-a)\right) = \int_0^\infty \frac{b\sinh(t)}{\cosh(t)(\frac{1-a}{2}\cosh(t)+\frac{1+a}{2})} \dx t.
	\end{equation}
	Indeed, the integrand has a primitive function
	$$\frac{2b}{1+a}\log\left(\frac{\cosh(t)}{\cosh(t)+\frac{1+a}{1-a}}\right),$$
	which shows \eqnref{eq:log}. Then we have, for $a>0$
	\begin{align*}
	\int_0^\infty \frac{\sinh(bt)}{\cosh(t)-a}\dx t + \frac{2b}{1+a}\log\left(\frac{1}{2}(1-a)\right) &= \int_0^\infty\left( \frac{\sinh(bt)}{\cosh(t)-a} - \frac{b\sinh(t)}{\cosh(t)(\frac{1-a}{2}\cosh(t)+\frac{1+a}{2})}\right)\dx t \\
	&> \int_0^\infty\frac{1}{\cosh(t)} \left(\sinh(bt) - \frac{b\sinh(t)}{\frac{1-a}{2}\cosh(t)+\frac{1+a}{2}}\right)\dx t \\
	&> 0,
	\end{align*}
	where the last step follows because $\sinh(bt) - \frac{b\sinh(t)}{\frac{1-a}{2}\cosh(t)+\frac{1+a}{2}}>0$ for all $t>0$ in the case $a>0$. This proves the lemma in the case $a>0$. We now turn to the case $a\leq0$. It is easy to see that for every $b$,
	\begin{equation}\label{eq:minmax}
	\min_{-1\leq a \leq 0}\int_0^\infty \frac{\sinh(bt)}{\cosh(t)-a}\dx t = \int_0^\infty \frac{\sinh(bt)}{\cosh(t) + 1}\dx t, \qquad \max_{-1\leq a \leq 0} -\frac{2b}{1+a}\log\left(\frac{1}{2}(1-a)\right) = 2b\log\left(2\right).
	\end{equation}
	Moreover, we have 
	$$\int_0^\infty \frac{\sinh(bt)}{\cosh(t) + 1}\dx t > \int_0^\infty \frac{bt}{\cosh(t) + 1}\dx t = 2b\log(2),$$
	where we have used a known value for the integral (for example found in \cite{table}). Together with \eqnref{eq:minmax}, this proves the lemma in the case $a\leq0$.
\end{proof} 

\noindent \textit{Proof of Theorem \ref{thm:band_gap}.} We will show that there is a frequency $\omega_0$ such that 
$$\max_{|\alpha|>\alpha_0} \omega_1^\alpha < \omega_0 < \min_{|\alpha|>\alpha_0} \omega_2^\alpha.$$
For sufficiently small $\delta$, by \Cref{thm:char_approx_infinite} and \Cref{lem:evec}, it is enough to show that there is a constant $C_0$ such that 
\begin{equation} \label{eq:Cgap}
\max_{|\alpha|>\alpha_0} C_{11}^\alpha - |C_{12}^\alpha| < C_0 < \min_{|\alpha|>\alpha_0} C_{11}^\alpha + |C_{12}^\alpha|.
\end{equation}
Define $C_0$ as
$$C_0 = \epsilon \text{Cap}_B - \frac{(\epsilon \text{Cap}_B)^2}{4\pi L}\sum_{m\neq 0} \frac{(-1)^m}{|m|},$$
that is, $C_0$ is defined as the leading order of the eigenvalues of $C^\alpha$ at the degenerate point $d=d', \alpha = \pi/L$. The sum appearing in the expansion of $C_{11}^\alpha$ can be explicitly computed as
$$\sum_{m\neq 0} \frac{e^{\iu m\alpha L}}{|m|}=-\log\big(2-2\cos(\alpha L)\big).$$
Then we have 
\begin{align*}
C_{11}^\alpha - |C_{12}^\alpha| - C_0 &= \frac{(\epsilon \text{Cap}_B)^2}{4\pi L} \log\left(\frac{1}{2}\big(1-\cos(\alpha L)\big)\right) - |C_{12}^\alpha| +O(\epsilon^3) \\
&< 0,
\end{align*}
for $\epsilon$ small enough. This shows that, for sufficiently small $\epsilon$,
$$\max_{|\alpha|>\alpha_0} C_{11}^\alpha - |C_{12}^\alpha| < C_0.$$
We now turn to the second inequality of \eqnref{eq:Cgap}. By \eqnref{eq:c12d} and \eqnref{eq:fint} we have
\begin{align} \label{eq:intermediate}
C_{11}^\alpha + |C_{12}^\alpha| - C_0 &= \frac{(\epsilon \text{Cap}_B)^2}{4\pi L} \left( \log\left(\frac{1}{2}\big(1-\cos(\alpha L)\big)\right) + \left|\int_0^\infty\frac{e^{\iu \alpha L}\sinh\left(\frac{d}{L}t\right) + \sinh\left(\frac{1-d}{L}t\right)}{\cosh(t) - \cos(\alpha L)}\dx t\right|\right)  +O(\epsilon^3).
\end{align}
Recall that $0<d<L$, so we can apply \eqnref{eq:ineq} with $b= d/L$ or with $b=(1-d)/L$ and with $a=\cos(\alpha L)$. Expanding the absolute value and applying \eqnref{eq:ineq}, we find after simplifications that
\begin{align*}
\left|\int_0^\infty\frac{e^{\iu \alpha L}\sinh\left(\frac{d}{L}t\right) + \sinh\left(\frac{1-d}{L}t\right)}{\cosh(t) - \cos(\alpha L)}\dx t\right|^2 &> \log\left(\frac{1}{2}\big(1-\cos(\alpha L)\big)\right)^2 \frac{4\left(d^2 + 2d(1-d)\cos(\alpha L) + (1-d)^2 \right)}{L^2\big(1+\cos(\alpha L)\big)^2}\\
& > \log\left(\frac{1}{2}\big(1-\cos(\alpha L)\big)\right)^2.
\end{align*}
Together with \eqnref{eq:intermediate}, this shows that 
$$\min_{|\alpha|>\alpha_0} C_{11}^\alpha + |C_{12}^\alpha| > C_0,$$
for $\epsilon$ small enough. We have thus proved \eqnref{eq:Cgap}, from which the theorem follows.
\qed
\end{appendices}
\end{document}